\documentclass[12pt, a4paper]{amsart}
\usepackage{amscd,amssymb,mathtools}
\usepackage[shortlabels]{enumitem}
\usepackage{tensor}
\usepackage{wrapfig}

\usepackage{tikz}
\usetikzlibrary{cd,calc,overlay-beamer-styles,decorations.pathmorphing,arrows.meta}

\tikzset{
    vertex/.style={circle, fill=black, scale=0.2, outer sep=2mm},
    open-vertex/.style={circle, draw=black, scale=0.2, outer sep=2mm},
    forwards/.style={every edge/.append style={-stealth,bend left=20}},
    backwards/.style={every edge/.append style={stealth-,dashed,bend right=20}},
}

\tikzstyle{goto}=[->,shorten >=1pt,>=stealth,semithick]

\allowdisplaybreaks

\usepackage{color, url, hyperref}

\usepackage{soul}
\usepackage{xcolor}

\usepackage{graphicx}

\usepackage{adjustbox}

\hypersetup{colorlinks=true, linkcolor=blue, linkbordercolor=blue, citecolor=blue, citebordercolor=blue, linktocpage=true}

\DeclareMathOperator{\coker}{coker}

\DeclareMathOperator{\Ad}{Ad}
\DeclareMathOperator{\id}{id}
\DeclareMathOperator{\sgn}{sgn}
\DeclareMathOperator{\Prob}{Prob}
\DeclareMathOperator{\diag}{diag}

\newcommand{\into}{\hookrightarrow}

\DeclarePairedDelimiter{\abs}{\lvert}{\rvert}
\DeclarePairedDelimiter{\norm}{\lVert}{\rVert}

\def\C{\mathbb{C}}

\def\N{\mathbb{N}}
\def\Z{\mathbb{Z}}

\def\Q{\mathbb{Q}}

\def\FF{\mathcal{F}}
\def\GG{\mathcal{G}}

\def\JJ{\mathcal{J}}
\def\KK{\mathcal{K}}

\def\VV{\mathcal{V}}

\newtheorem{thm}{Theorem}[section]
\newtheorem{lemma}[thm]{Lemma}
\newtheorem{prop}[thm]{Proposition}
\newtheorem{coro}[thm]{Corollary}

\theoremstyle{definition}
\newtheorem{defin}[thm]{Definition}
\newtheorem{eg}[thm]{Example}

\theoremstyle{remark}
\newtheorem{remark}[thm]{Remark}
\newtheorem{notation}[thm]{Notation}

\numberwithin{equation}{section}

\begin{document}

\date{\today}
\title[Group actions on multitrees and the $K$-theory of their crossed products]{Group actions on multitrees and the \\ $K$-theory of their crossed products}

\author[Brownlowe]{Nathan Brownlowe}
\author[Spielberg]{Jack Spielberg}
\author[Thomas]{Anne Thomas}
\author[Wu]{Victor Wu}

\address{Nathan Brownlowe, Anne Thomas and Victor Wu \\ School of Mathematics and
	Statistics  \\
	The University of Sydney} \email{nathan.brownlowe@sydney.edu.au, anne.thomas@sydney.edu.au}\email{viwu8694@uni.sydney.edu.au}
\address{Jack Spielberg, School of Mathematical and Statistical Sciences\\Arizona State University\\USA} \email{jack.spielberg@asu.edu}


\subjclass[2010]{}
\thanks{}

\begin{abstract}
	We study group actions on multitrees, which are directed graphs in which there is at most one directed path between any two vertices. In our main result we describe a six-term exact sequence in $K$-theory for the reduced crossed product $C_0(\partial E)\rtimes_r G$ induced from the action of a countable discrete group $G$ on a row-finite, finitely-aligned multitree $E$ with no sources. We provide formulas for the $K$-theory of $C_0(\partial E) \rtimes_r G$ in the case where $G$ acts freely on $E$, and in the case where all vertex stabilisers are infinite cyclic. We study the action $G\curvearrowright  \partial E$ in a range of settings, and describe minimality, local contractivity, topological freeness, and amenability in terms of properties of the underlying data. In an application of our main theorem, we describe a six-term exact sequence in $K$-theory for the crossed product induced from a group acting on the boundary of an undirected tree. 
\end{abstract}

\maketitle

\setcounter{tocdepth}{1}


\section{Introduction}\label{sec:intro}

Group actions on (undirected) trees have for many years provided interesting objects of study from the point of view of $C^*$-algebras. While the groups themselves give important case studies in $C^*$-simplicity \cite{dlHarpe-Preaux} and amenability \cite{Julg-Valette}, the action is best studied via a crossed product $C^*$-algebra. While crossed products by groups acting on trees have been studied in general \cite{Pimsner}, the induced action of the group on the \textit{boundary} of the tree on which it acts has been particulary fertile ground for crossed products, and examples such as the Cuntz--Krieger algebras of \cite{Spielberg,Robertson}, and the Bunce--Deddens algebras of \cite{Orfanos}, have benefited from a crossed-product model arising from such boundary actions.

Group actions on trees are also central to the field of geometric group theory. Here, they are studied using Bass--Serre theory, which establishes a bijection between group actions on trees and quotient structures called graphs of groups \cite{Serre, Bass}. In \cite{BMPST}, the authors established a $C^*$-algebraic Bass--Serre theory by producing a graph-of-groups $C^*$-algebra which is Morita equivalent to the crossed product induced from the corresponding  boundary action of the tree. Using this description, the authors were able to study free group actions on trees, and group actions with infinite cyclic stabiliser groups. The work in \cite{BMPST} follows earlier work on $C^*$-algebras associated to graphs of groups by Okayasu \cite{Okayasu} and undirected graphs (graphs of trivial groups) by Cornelissen, Lorscheid and Marcolli \cite{CLM}. Following \cite{BMPST}, Mundey and Rennie examined the $K$-theory of graph-of-groups $C^*$-algebras by first producing a Cuntz--Pimsner model for graph-of-groups $C^*$-algebras \cite{Mundey-Rennie}.

In this paper we study a more general type of action. A \textit{multitree} is a directed graph $E$ in which there is at most one directed path between any two vertices; these objects were introduced in \cite{multitrees}, and they have been studied in fields relating to computer science, such as data representation and complexity theory. We consider actions of a countable discrete group $G$ on a row-finite, finitely-aligned multitree $E$ which has no sources. Such an action $G\curvearrowright E$ induces an action of $G$ on the boundary $\partial E$, which is the totally disconnected locally compact Hausdorff space of shift-tail equivalent infinite paths in $E$. In this article we are interested in the problem of calculating the $K$-theory of the reduced crossed product $C_0(\partial E)\rtimes_r G$.

Our approach to this problem starts with the observation that under certain assumptions (such as $G$ being exact, or the action of $G$ on $E^0 \cup \partial E$ being amenable), we have a short exact sequence 
\[
\begin{tikzcd}
	0 \ar{r} &C_0(E^0)\rtimes_r G \arrow[r,"\iota"] & C_0(E^0 \cup \partial E)\rtimes_r G \ar{r} & C_0(\partial E) \rtimes_r G  \ar{r} &0
\end{tikzcd}
\]
giving rise to the six-term exact sequence in $K$-theory
\[\begin{tikzcd}
	K_0(C_0(E^0) \rtimes_r G)
	\arrow[r,"\iota_{*,0}"] &
	K_0(C_0(E^0 \cup \partial E) \rtimes_r G) \arrow[r] &
	K_0(C_0(\partial E) \rtimes_r G) \arrow[d] \\
	K_1(C_0(\partial E) \rtimes_r G) \arrow[u] &
	K_1(C_0(E^0 \cup \partial E) \rtimes_r G) \arrow[l] &
	K_1(C_0(E^0) \rtimes_r G).
	\arrow[l,"\iota_{*,1}"]
\end{tikzcd}\]
We then turn our attention to the $K$-groups $K_*(C_0(E^0) \rtimes_r G)$ and $K_*(C_0(E^0 \cup \partial E) \rtimes_r G)$. For this we use the results of Cuntz, Echterhoff and Li \cite{CEL}, in which they prove that if $G$ is a locally compact group acting on a totally disconnected locally compact space $\Omega$, and there is a $G$-invariant collection $\mathcal{V}=\{V_i:i\in I\}$ of compact open subsets of $\Omega$ satisfying certain properties, then the $K$-theory of the reduced crossed product $C_0(\Omega)\rtimes_r G$ decomposes as a direct sum of the $K$-theories of reduced group $C^*$-algebras
\begin{equation}\label{eq:celktheoryformula}
	K_*(C_0(\Omega)\rtimes_r G)\cong \bigoplus_{[i]\in G\setminus I}K_*(C^*_r(G_i)).
\end{equation}
In \cite{CEL} it is assumed that $\mathcal{V}$ is closed under intersections, which causes a problem in the setting of multitrees. The collection of cylinder sets $\{Z(v):v\in E^0\}$ forms a natural collection of compact open sets in $\Omega:=E^0\cup\partial E$ which is not closed under intersections, but is rather finitely aligned, in the sense that the intersection $Z(v)\cap Z(w)$ is a finite disjoint union of other cylinder sets. To overcome this problem we generalise (in the setting of discrete group actions) the results of \cite{CEL} to the setting of finitely-aligned $\mathcal{V}$, and in particular, we show that the decomposition in Equation~\eqref{eq:celktheoryformula} still holds in this more general setting. During the preparation of this paper, we learnt that Alistair Miller also generalised \cite{CEL} in his PhD thesis \cite{Miller-PhD}.

Using our generalised version of Equation~\eqref{eq:celktheoryformula} we are able to prove our main result: the six-term exact sequence in $K$-theory
\[
\begin{tikzcd}
	\displaystyle\bigoplus_{[v]\in G\setminus E^0}K_0(C^*_r(G_v)) \ar{r} & \displaystyle\bigoplus_{[v]\in G\setminus E^0}K_0(C^*_r(G_v)) \ar{r} & K_0(C(\partial E) \rtimes_r G) \ar{d} \\
	K_1(C(\partial E) \rtimes_r G)  \ar{u} & \displaystyle\bigoplus_{[v]\in G\setminus E^0}K_1(C^*_r(G_v)) \ar{l} & \displaystyle\bigoplus_{[v]\in G\setminus E^0}K_1(C^*_r(G_v)), \ar{l}
\end{tikzcd}
\]
where we give concrete formulae for the maps involved, which are induced from homomorphisms on the level of the underlying $C^*$-algebras.

We apply this six-term exact sequence to produce specific $K$-theory formulae for two classes of group actions on multitrees. We show that when $G$ acts freely on $E$, then 
\[K_0(C_0(\partial E) \rtimes_{\tau_\partial,r} G) \cong \coker(1-A^T) \quad \text{and} \quad K_1(C_0(\partial E) \rtimes_{\tau_\partial,r} G) \cong \ker(1-A^T),\]
where $A$ is the adjacency matrix of the quotient graph $G\backslash E$. In the case where the action $G\curvearrowright E$ produces infinite cyclic vertex stabiliser groups, and that $G\curvearrowright (E^0 \cup\partial E)$ is amenable, we show that
\[ K_0(C_0(\partial E) \rtimes_{r} G) \cong \coker(1-A_0) \oplus \ker(1-A_1) \]
and
\[ K_1(C_0(\partial E) \rtimes_{r} G) \cong \coker(1-A_1) \oplus \ker(1-A_0), \]
where $A_0$ and $A_1$ are integer-valued matrices depending on the transversal data obtained from the inclusions of edge stabilisers into vertex stabilisers. This can be applied for example to generalised Baumslag--Solitar groups acting on the boundary of directed trees.

Our work was initially motivated by the problem of calculating the $K$-theory of the crossed products arising from group actions on undirected trees, and we are able to use our results to answer this question. The link between multitrees and undirected trees is given by the dual graph construction: if $T$ is an undirected tree, then the dual graph $\widehat{T}$ is a multitree, and there is a natural isomorphism of crossed products $C_0(\partial T)\rtimes_r G\cong C_0(\partial\widehat{T})\rtimes_r G$. In the setting where all vertex stabilisers are amenable, this means we have a six-term exact sequence for the $K$-theory of the graph-of-groups $C^*$-algebras of \cite{BMPST}, and it connects our results to the $K$-theoretic results of Mundey and Rennie for graph-of-groups $C^*$-algebras given in \cite{Mundey-Rennie}. We do note that the two approaches are different; while in \cite{Mundey-Rennie} the maps in their six-term exact sequence are described in terms of the $KK$-classes of $C^*$-correspondences, we are able to give the more direct approach of having induced homomorphisms as the maps in our sequence. It is clear that in the setting of a free action, or when the vertex stabilisers are infinite cyclic, our results coincide. In more generality it is not so clear, and it would be an interesting problem to clarify the precise relationship between the two approaches. 

We start with preliminaries in Section~\ref{sec:prelim}, where we give the necessary definitions and results for group actions on graphs. We also recall the definition of the boundary of a locally finite, nonsingular tree, and its topology. In Section~\ref{sec:celktheory} we look at the results of \cite{CEL} on the $K$-theory of crossed products induced by group actions on totally disconnected spaces, and we generalise some of their results in the setting of discrete group actions. The main tool to come out of work in this section is the generalised version of Equation~\eqref{eq:celktheoryformula}. Section~\ref{sec:multitreektheory} is devoted to the main results of the paper, on the $K$-theory of the crossed product $C_0(\partial E)\rtimes_r G$ coming from a discrete group $G$ acting on a row-finite, finitely aligned multitree $E$ with no sources. We consider two classes of examples: when the action of $G$ on $E$ is free, and hence the vertex stabilisers are trivial; and when all vertex stabilisers are infinite cyclic. In both settings we give an explicit formula for the $K$-theory of $C_0(\partial E)\rtimes_r G$ in terms of the kernel and cokernel of a matrix (or matrices) associated to the underlying data. In Section~\ref{sec:undirectedtrees} we apply our $K$-theory results for multitrees to the dual of an undirected tree $T$ to get an analogous six-term exact sequence in $K$-theory for the crossed product from an action on the boundary of an undirected tree. 

We devote Section \ref{sec:propsofaction}, our final section, to studying dynamical properties of actions on multitrees. It is well known that dynamical properties of a discrete group $G$ acting on a locally compact Hausdorff space $X$ induce key structural properties of the corresponding reduced crossed product $C_0(X)\rtimes_r G$. For example, $C_0(X)\rtimes_r G$ is simple if the action is minimal and topologically free; a simple $C_0(X)\rtimes_r G$ is purely infinite if the action is locally contractive; and the full and reduced crossed products coincide if the action is amenable. For a countable discrete group $G$ acting on a multitree $E$, we:
\begin{itemize}
	\item characterise minimality of $G\curvearrowright \partial E$ in terms of the quotient directed graph $\Gamma:=G\backslash E$;
	\item provide a sufficient condition on $\Gamma$ and the stabiliser subgroups that ensures that $G \curvearrowright \partial E$ is locally contractive; 
	\item show that if $G\curvearrowright E^0$ is free, then $G\curvearrowright\partial E$ is topologically free if and only if $\Gamma$ is aperiodic;
	\item characterise topological freeness of $G\curvearrowright \partial E$ in terms of $\Gamma$ and the stabiliser subgroups in the case where these stabiliser subgroups are infinite cyclic; and
	\item in the case where $E$ is a directed tree $T_+$, show that if each stabiliser subgroup is amenable, then the action $G\curvearrowright \partial T_+$ is amenable.
\end{itemize}

\vspace{0.2cm}

\textbf{Acknowledgements.} The first author was supported by the Australian Research Council grant DP200100155. The last author was supported by an Australian Government Research Training Program Stipend Scholarship.

\section{Preliminaries on group actions on graphs}\label{sec:prelim}

\subsection{Graphs and trees}\label{subsec:graphsandtrees}

An \textit{undirected graph} $\Gamma=(\Gamma^0, \Gamma^1, r, s)$ (called simply a \textit{graph} in \cite{BMPST, Mundey-Rennie, Serre}) consists of countable sets of \textit{vertices} $\Gamma^0$ and \textit{edges} $\Gamma^1$ together with \textit{range} and \textit{source maps} $r, s \colon \Gamma^1 \to \Gamma^0$, as well as an involutive ``edge reversal" map $e \mapsto \overline{e}$ on $\Gamma^1$. We insist that $e \neq \overline{e}$ and $r(e) = s(\overline{e})$. We say that an undirected graph $\Gamma$ is \textit{locally finite} if $\abs{r^{-1}(v)} < \infty$ for all $v \in \Gamma^0$, and \textit{nonsingular} if $\abs{r^{-1}(v)} > 1$ for all $v \in \Gamma^0$. All undirected graphs in this paper will be locally finite and nonsingular.

A \textit{path (of length $n$)} in $\Gamma$ is either a vertex $v \in \Gamma^0$ (when $n=0$) or, if $n > 0$, a sequence of edges $e_1 e_2 \dots e_n$ with $s(e_i) = r(e_{i+1})$ for $1 \leq i \leq n-1$ (we note that we are using the right-to-left ``Australian" convention for notating paths, to be consistent with \cite{BMPST} and \cite{Mundey-Rennie}). A path $e_1 e_2 \dots e_n$ is \textit{reduced} if either $n=0$, or $e_{i+1} \neq \overline{e_i}$ for $1 \leq i \leq n-1$ (that is, there is no back-tracking in the path). We write $\Gamma^n$ (respectively, $\Gamma^n_r$) for the collection of all paths (respectively, all \textit{reduced} paths) of length $n$ in $\Gamma$, and we write $\Gamma^*_r$ for the collection of all reduced paths of any length in $\Gamma$. For any vertex $v \in \Gamma^0$, we write $v\Gamma^*_r$ (respectively $v\Gamma^n_r$) for the collection of reduced paths in $\Gamma$ (respectively, the collection of reduced paths of length $n$) with range $v$, and we write $\Gamma^*_rv$ (respectively $\Gamma^n_rv$) for the collection of reduced paths in $\Gamma$ (respectively, the collection of reduced paths of length $n$) with source $v$.

We call a graph $\Gamma$ \textit{connected} if for any two vertices $v, w \in \Gamma^0$ there is a path in $\Gamma$ between them. We call $\Gamma$ a \textit{tree} if for any two vertices $v, w \in \Gamma^0$, there is a unique reduced path from $v$ to $w$.

We will also be considering \textit{directed} graphs. A \textit{directed graph} $E=(E^0, E^1, r, s)$ consists of countable sets of vertices $E^0$ and edges $E^1$ together with range and source maps $r, s \colon E^1 \to E^0$, but without an edge reversal map. Equivalently, we can think of a directed graph $E$ as an undirected graph $\Gamma$ (called the \textit{underlying undirected graph} of $E$) equipped with an \textit{orientation}, which is a subset $\Gamma_+^1$ of $\Gamma^1$ which contains exactly one element of $\{e, \overline{e}\}$ for each $e \in \Gamma^1$; this orientation specifies a direction for each edge in the graph. A \textit{directed tree} is a directed graph whose underlying undirected graph is a tree.

Paths in a directed graph are defined and notated in the same way as for undirected graphs. If we form a directed graph from an orientation on an undirected graph, then paths can only consist of edges in the orientation. Note however that the concept of a reduced path is redundant for directed graphs, since the reversal of an edge in an orientation cannot be an element of that orientation. We will therefore omit the subscript `$r$' in the $\Gamma^*_r$ notation when $\Gamma$ is directed.

We say that a directed graph $E$ is \textit{row-finite} if $\abs{r^{-1}(v)} < \infty$ for all $v \in E^0$, and we say that a vertex $v \in E^0$ is a \textit{source} if $\abs{r^{-1}(v)} = 0$. All directed graphs in this paper will be row-finite and have no sources.

\subsection{Group actions on graphs}\label{subsec:groupactionsongraphs}

A group $G$ is said to \textit{act} (on the left) on a graph $\Gamma$ (undirected or directed) if it acts on the sets $\Gamma^0$ and $\Gamma^1$ in such a way that $r(g \cdot e) = g \cdot r(e)$, $s(g \cdot e) = g \cdot s(e)$ and (for undirected $\Gamma$) $\overline{g \cdot e} = g \cdot \overline{e}$ for all $e \in \Gamma^1$ and $g \in G$. For each $v \in \Gamma^0$, we write $G_v \leq G$ for the stabiliser subgroup of $G$ at $v$, and similarly for each $e \in \Gamma^1$ we write $G_e \leq G$ for the stabiliser subgroup of $G$ at $e$. These stabiliser subgroups have the following properties:

\begin{enumerate}[(1)]
	\item If $v, w \in \Gamma^0$ are in the same orbit under the action of $G$, then the subgroups $G_v$ and $G_w$ are conjugate. Indeed, if $g \in G$ satisfies $g \cdot v = w$, then $G_w = g G_v g^{-1}$. Similarly, if $e, f \in \Gamma^1$ are in the same orbit under the action of $G$, then the subgroups $G_e$ and $G_f$ are conjugate.
	
	\item For any $e \in \Gamma^1$, the group $G_e$ is a subgroup of both $G_{r(e)}$ and $G_{s(e)}$. This is because any automorphism of $\Gamma$ fixing $e$ must also fix both the vertices $r(e)$ and $s(e)$.
	
	\item By the orbit-stabiliser theorem, for any $e \in \Gamma^1$, the orbit $G_{r(e)}\cdot e$ is in one-to-one correspondence with the set of cosets $G_{r(e)}/G_e$. In particular, if $\Gamma$ is locally finite or row-finite, then $G_e$ is a finite-index subgroup of $G_{r(e)}$.
\end{enumerate}

If a group $G$ acts on the directed graph $E$, then we can also consider the \textit{quotient directed graph} $G\backslash E$ associated to this action, which is defined as follows. The vertices $(G\backslash E)^0$ are the vertex orbits $\{G \cdot v : v \in E^0\}$; the edges $(G\backslash E)^1$ are the edge orbits $\{G \cdot e : e \in E^1\}$; and the range and source maps are given by $r(G \cdot e) = G \cdot r(e)$ and $s(G \cdot e) = G \cdot s(e)$, respectively, for $e \in E^1$.

\subsection{Boundaries}\label{subsec:boundaries}

We recall the definition of the boundary of a locally finite, nonsingular tree.

An \textit{infinite path} in a graph $\Gamma$ (undirected or directed) is an infinite sequence of edges $e_1 e_2 \dots$ such that $s(e_i) = r(e_{i+1})$ for all $i \geq 1$. Like for finite paths, we call an infinite path $e_1 e_2 \dots$ \textit{reduced} if $e_{i+1} \neq \overline{e_i}$ for $i \geq 1$. We write $\Gamma^\infty$ (respectively, $\Gamma^\infty_r$) for the collection of all infinite paths (respectively, all reduced infinite paths) in $\Gamma$. For any $v \in \Gamma^0$, we write $v\Gamma^\infty_r$ for the collection of all reduced infinite paths with range $v$. For any infinite path $\lambda = e_1 e_2 \dots$ in $\Gamma$ and any integer $n \geq 1$, we write $\sigma^n(\lambda)$ for the infinite path $e_{n+1} e_{n+2} \dots$, and $\lambda_n$ for the finite path $e_1 \ldots e_n$.

If $T$ is a locally finite, nonsingular tree, then we can define its \textit{boundary} as follows. Define an equivalence relation on $T^\infty_r$ by $\lambda \sim \mu$ if and only if $\sigma^m(\lambda) = \sigma^n(\mu)$ for some $m,n \in \N$ (in this case we say that $\lambda$ and $\mu$ are \textit{shift-tail equivalent}). The \textit{boundary} of $T$ is the set of equivalence classes of reduced infinite paths in $T$, and is denoted $\partial T$. For $\lambda \in T^\infty_r$ we write $[\lambda]$ for its shift-tail equivalence class in $\partial T$.

We can put a topology on $\partial T$ in the following way. For a path $\alpha \in T^*_r$ of length $n$, we say that an infinite path $\lambda \in T^\infty_r$ \textit{extends} $\alpha$ if $\lambda_n = \alpha$, and we define the \textit{cylinder set} $Z_\partial(\alpha)$ to be the subset $\{[\lambda] : \lambda \in T^\infty_r\ \text{extends}\ \alpha\}$ of $\partial T$. The collection $\{Z_\partial(e) : e \in T^1\}$ is a base for a topology on $\partial T$, turning it into a totally disconnected compact Hausdorff space. 

If $T$ is directed, then the boundary is defined as above, and now $\{Z_\partial(v) : v \in T^0\}$ is a base for a totally disconnected locally compact Hausdorff topology.

An action of a group $G$ on a tree $T$ naturally induces an action of the same group on the boundary of tree, by $g \cdot [\lambda] = [g \cdot \lambda]$ for all $\lambda \in T^\infty_r$. In this paper, we are interested in understanding actions arising in this way.

\section{$K$-Theory of crossed products by actions on totally disconnected spaces}\label{sec:celktheory}

In this section, we generalise the $K$-theory results of \cite[Section~3]{CEL} so that we can apply them to the setting of group actions on multitrees in Section~\ref{sec:multitreektheory}.

Let $G$ be a second countable locally compact group acting continuously on a second countable totally disconnected locally compact space $\Omega$, with $\mathcal{V} = \{V_i : i \in I\}$ a $G$-invariant family of compact open sets in $\Omega$. The $G$-invariance of $\mathcal{V}$ means that $G$ acts on $I$ with $g\cdot V_i=V_{g\cdot i}$ for each $g\in G,i\in I$. Denote by $\tau$ the induced action of $G$ on $C_0(\Omega)$, and suppose that $\alpha$ is an action of $G$ on a separable $C^*$-algebra $A$. Under some additional assumptions, in \cite[Corollary~3.14]{CEL} Cuntz, Echterhoff and Li prove that there is an isomorphism
\[
K_*((A\otimes C_0(\Omega))\rtimes_{\alpha\otimes\tau,r} G) \cong \bigoplus_{[i] \in G\backslash I} K_*(A\rtimes_{\alpha,r}G_i),
\]
where $G_i \leq G$ is the stabiliser subgroup at $i \in I$. Our main result in this section is Theorem~\ref{thm:mainCELthm}, which establishes this isomorphism in a more general setting than the one considered in \cite{CEL}.

We start with a discussion of this more general setting in Section~\ref{subsec:finite alignment}, and in Section~\ref{subsec:results on projections} we prove a technical result which we then use in Section~\ref{subsec:CEL generalisation} to generalise the $K$-theoretic arguments in \cite{CEL}.

\subsection{Finite alignment} \label{subsec:finite alignment}
Let $X$ be a set, and $\VV$ a subset of the power set $\mathcal{P}(X)$. Following \cite[Definition~2.4]{CEL} we say $\VV$ is \textit{independent} if for all $V,V_1\dots,V_k \in \VV$ we have $V=\cup_{i=1}^kV_k\implies V=V_i$ for some $1\le i\le k$. We say $\VV$ is \textit{finitely aligned} if for all $V_i,V_j\in\VV$, the intersection $V_i\cap V_j$ can be written as a (possibly empty) finite disjoint union of other sets in $\VV$.

In this subsection, we study the setting of a group $G$ acting on a totally disconnected space $\Omega$ which admits a $G$-invariant, independent, finitely aligned generating family of compact open subsets. Our main result here is Proposition~\ref{prop:cofinality}, which is a technical result that we will use to prove Theorem~\ref{thm:mainCELthm}.

\begin{defin}
    Let $H$ be any subgroup of $G$.
    \begin{enumerate}[(1)]
        \item For any $J\subseteq I$ define $\VV_J:=\{V_i:i\in J\}$.
        \item Define $\FF_H:= \{F\subseteq I: F\text{ is finite and $H$-invariant}\}$.
        \item Define $\JJ_H := \{J\in \FF_H : \VV_J\text{ is finitely aligned}\}$.
    \end{enumerate}
\end{defin}

\begin{prop}\label{prop:cofinality}
	Let $G$ be a second countable locally compact group acting continuously on a second countable totally disconnected locally compact space $\Omega$. Let $I$ be a set, and let $\VV:=\{V_i : i \in I\}$ be a $G$-invariant collection of compact open subsets of $\Omega$ such that:
    \begin{enumerate}[(1)]
        \item $\VV$ generates the compact open sets of $\Omega$ (via finite unions, finite intersections, and differences);
        \item $\VV$ is independent; and 
        \item $\VV$ is finitely aligned.
    \end{enumerate}
    The set $\JJ_H$ is a cofinal subset of $(\FF_H,\subseteq)$. 
\end{prop}

\begin{remark}
    Let $H$ be any compact subgroup of $G$. These assumptions on $\VV$ in the statement of Proposition~\ref{prop:cofinality} generalise the notion of a \textit{regular basis} in \cite[Definition~2.9]{CEL}, where instead of finite alignment (3), the authors assume that $\VV$ is closed under intersections (up to the empty set).
\end{remark}

To prove Proposition~\ref{prop:cofinality} we need the following lemma.

\begin{lemma}\label{lem:uniquenessofintersections}
	Let $F\subseteq I$ be finite. Then $\bigcap_{i\in F}V_i$ decomposes uniquely as a finite disjoint union of sets in $\VV$.
\end{lemma}

\begin{proof}
	A standard inductive argument shows that $\VV$ finitely aligned means that the intersection of finitely many elements of $\VV$ decomposes as a (possibly empty) finite disjoint union of sets in $\VV$. So we just need to prove uniqueness. Suppose that
	\begin{equation}\label{eq:twodisjointunions}
		\bigcap_{i\in F}V_i = \bigsqcup_{i=1}^m W_i = \bigsqcup_{j=1}^n W'_j,
	\end{equation}
	where $m,n \geq 1$ and each $W_i, W'_j \in \mathcal{V}$. We want to prove that
    \[ \{W_1, \dots, W_m\} = \{W'_1, \dots, W'_n\}. \]
	
	Fix $i \in \{1, \dots, m\}$. By Equation~\eqref{eq:twodisjointunions}, we have that
	\[
	W_i = \bigsqcup_{j=1}^n (W_i \cap W'_j),
	\]
	where each $W_i \cap W'_j$ is a disjoint union of sets in $\mathcal{V}$. Then the independence of $\mathcal{V}$ implies that $W_i \subseteq W_i \cap W'_j$ for some $j$, which in turn implies that $W_i \subseteq W'_j$. By symmetry, there is some $k \in \{1, \dots, m\}$ such that $W'_j \subseteq W_k$, and so $W_i \subseteq W_k$. But then disjointness of the sets $W_1, \dots, W_m$ implies that we must have $i = k$, and so $W_i = W'_j$. This shows that $\{W_1, \dots, W_m\} \subseteq \{W'_1, \dots, W'_n\}$, and a symmetric argument gives equality.
\end{proof}

We are now able to prove Proposition~\ref{prop:cofinality}.

\begin{proof}[Proof of Proposition~\ref{prop:cofinality}]
	Fix $F\in\FF_H$. For each $Y\subseteq F$ we know from Lemma~\ref{lem:uniquenessofintersections} that there is a unique (finite) $J_Y\subseteq I$ with $\bigcap_{i\in Y}V_i=\bigsqcup_{j\in J_Y}V_j$. We claim that $J:=\bigcup_{Y\subseteq F}J_Y$ satisfies $F\subseteq J$ and $J\in \JJ_H$.
	
	To see that $F\subseteq J$, let $i\in F$ and $Y=\{i\}$. Then $J_Y=\{i\}$, and so $i\in J$. We now show that $J\in\JJ_H$. We first claim that $J$ is $H$-invariant. Fix $j\in J$ and $h\in H$. Let $Y\subseteq F$ with $j\in J_Y$, and consider $\bigcap_{i\in h\cdot Y}V_i$. Since $F$ is $H$-invariant, we know that $h\cdot Y\subseteq F$, and so 
	\[
	\bigcap_{i\in h\cdot Y}V_i=\bigsqcup_{k\in J_{h\cdot Y}}V_k.
	\]
	We also have 
	\[
	\bigcap_{i\in h\cdot Y}V_i=\bigcap_{i\in Y}V_{h\cdot i}=\bigcap_{i\in Y}h\cdot V_i=h\cdot\left(\bigcap_{i\in Y}V_i\right)=h\cdot\left(\bigsqcup_{k\in J_Y}V_k\right)=\bigsqcup_{k\in J_Y}h\cdot V_k.
	\]
	This gives two decompositions of $\bigcap_{i\in h\cdot Y}V_i$ as disjoint unions of sets in $\VV$; since by Lemma~\ref{lem:uniquenessofintersections} such a decomposition is unique, the set $h\cdot V_j$ appearing in the second decomposition must be equal to $V_k$ for some $k\in J_{h\cdot Y}$. Hence $h\cdot  j=k\in  J$, and so $J$ is $H$-invariant. Since $J$ is finite (being a finite union of finite sets), we have $J\in \FF_H$.
	
	To prove that $\VV_J$ is finitely aligned, we need to show that for any $i,j \in J$, there is $J'\subseteq J$ such that $V_i\cap V_j=\bigsqcup_{k\in J'}V_k$.
	
	Let $Y,Y'\subseteq F$ with $i\in J_Y$ and $j\in J_{Y'}$, and consider $\bigcap_{k\in Y\cup Y'}V_k$. We have 
	\[
	\bigcap_{k\in Y\cup Y'}V_k=\bigsqcup_{l\in J_{Y\cup Y'}}V_l.
	\]
	We also have
	\[
	\bigcap_{k\in Y\cup Y'}V_k =\left(\bigcap_{m\in Y}V_m\right)\cap\left(\bigcap_{m'\in Y'}V_{m'}\right)=\left(\bigsqcup_{n\in J_Y}V_n\right)\cap\left(\bigsqcup_{n'\in Y'}V_{n'}\right)=\bigsqcup_{n\in J_Y,n'\in J_{Y'}}(V_n\cap V_{n'}).
	\]
	For each $n\in J_Y$ and $n'\in J_{Y'}$, finite alignment ensures that there is a set $Y_{n,n'}\subseteq I$ with $V_n\cap V_{n'}=\bigsqcup_{p\in Y_{n,n'}}V_p$, so
    \[ \bigcap_{k\in Y\cup Y'}V_k = \bigsqcup_{\substack{n\in J_Y,n'\in J_{Y'} \\ p \in Y_{n,n'}}} V_p. \]
    So we have two decompositions of $ \bigcap_{k\in Y\cup Y'}V_k$ as disjoint unions of sets in $\VV$, and since by Lemma~\ref{lem:uniquenessofintersections} such a decomposition is unique, we get that
    \[ V_i\cap V_j = \bigsqcup_{p \in Y_{i,j}} V_p = \bigsqcup_{l\in J'\subseteq J_{Y\cup Y'}} V_l \]
    for some subset $J'\subseteq J_{Y\cup Y'}\subseteq J$. Hence $\VV_J$ is finitely aligned. 
\end{proof}

\subsection{Results on projections} \label{subsec:results on projections}

We know from \cite[Lemmas 2.2 and 2.3]{CEL} that studying sets of compact open subsets of a totally disconnected space which generate all compact open subsets is equivalent to studying sets of projections generating a commutative $C^*$-algebra. This subsection is dedicated to proving Proposition~\ref{prop:CEL iff}, which is a generalisation of \cite[Lemma~3.8]{CEL}, and uses the language of projections and $C^*$-algebras.

\begin{notation}
For $\{e_1,\dots,e_n\}$ a set of projections, for each $1\le i\le n$ we define 
\[
e'_i := e_i - \bigvee_{e_j < e_i} e_j.
\]
\end{notation}

\begin{prop} \label{prop:CEL iff}
	Let $\{e_1, \dots, e_n\}$ be a finite set of commuting projections generating a $C^*$-algebra $D$.
	\begin{enumerate}[(1)]
		\item The set $\{e'_1, \dots, e'_n\}$ is a family of non-zero, pairwise orthogonal projections linearly spanning $D$ if and only if $\{e_1, \dots, e_n\}$ is independent and $e_i e_j = \bigvee_{e_k \leq e_i e_j} e_k$ for any $i, j \in \{1, \dots, n\}$.
		\item In this case, the transition matrix $\Gamma=(\gamma_{ij})$ determined by the equations $e_j=\sum_{i=1}^n\gamma_{ij}e_i'$ is a unipotent, $\{0,1\}$-matrix, and hence invertible over $\mathbb{Z}$.
	\end{enumerate}
\end{prop}

To prove Proposition~\ref{prop:CEL iff}, we need the following lemma.

\begin{lemma} \label{lemma:CEL iff lemma}
	Let $\{e_1, \dots, e_n\}$ be a finite set of commuting projections generating a $C^*$-algebra $D$. Suppose that $\{e'_1, \dots, e'_n\}$ is a family of non-zero, pairwise orthogonal projections linearly spanning $D$. Then $e'_a \leq e_i\iff e_a \leq e_i$, for any $a, i \in \{1, \dots, n\}$.
\end{lemma}

\begin{proof}
	The implication $e_a\le e_i\implies e'_a\le e_i$ is immediate. For the reverse implication, note that we can write each $e_i$ as $e_i=\sum_{a=1}^n\gamma_{ai}e_a'$, where for each $a=1, \dots, n$, we have $\gamma_{ai} = 1$ if $e'_a \leq e_i$ and $\gamma_{ai} = 0$ otherwise. Pairwise orthogonality of the $e'_i$ means that for each pair of distinct $i, j \in \{1, \dots, n\}$, we have
	\[
	0 = e'_i e'_j = \left( e_i - \bigvee_{e_k < e_i} e_k \right) \left(e_j - \bigvee_{e_l < e_j} e_l \right) 
	= e_i e_j - \bigvee_{e_l < e_j} e_i e_l - \bigvee_{e_k < e_i} e_j e_k + \bigvee_{\substack{e_k < e_i \\ e_l < e_j}} e_k e_l,
	\]
	and so
	\begin{align*}
		e_i e_j &= \bigvee_{e_l < e_j} e_i e_l + \bigvee_{e_k < e_i} e_j e_k - \bigvee_{\substack{e_k < e_i \\ e_l < e_j}} e_k e_l \\
		&= \bigvee_{e_l < e_j} \left(\sum_{a=1}^n \gamma_{ai}\gamma_{al}e'_a\right) + \bigvee_{e_k < e_i} \left(\sum_{a=1}^n \gamma_{aj}\gamma_{ak}e'_a\right) - \bigvee_{\substack{e_k < e_i \\ e_l < e_j}} \left(\sum_{a=1}^n \gamma_{ak}\gamma_{al}e'_a\right).
	\end{align*}
	If we define $\mu_{ai} := \max\{\gamma_{ak} : e_k < e_i\}$ for all $a, i \in \{1, \dots, n\}$, then
    \begin{equation} \label{eq:CEL iff lemma}
        e_i e_j = \sum_{a=1}^n \gamma_{ai}\mu_{aj}e'_a + \sum_{a=1}^n \gamma_{aj}\mu_{ai}e'_a - \sum_{a=1}^n \mu_{ai}\mu_{aj}e'_a 
        = \sum_{a=1}^n (\gamma_{ai}\mu_{aj} + \gamma_{aj}\mu_{ai} - \mu_{ai}\mu_{aj})e'_a.
    \end{equation}
	
	Now suppose that $e'_a \leq e_i$. If $a = i$ then $e_a \leq e_i$ trivially, so suppose that $a \neq i$. Multiplying both sides of the inequality $e'_a \leq e_i$ by $e_a$ gives that $e'_a = e_a e'_a \leq e_a e_i$. From Equation~\eqref{eq:CEL iff lemma}, this implies that
	\[\gamma_{aa}\mu_{ai} + \gamma_{ai}\mu_{aa} - \mu_{aa}\mu_{ai} = 1.\]
	Now, $\gamma_{aa} = 1$ by definition, since $e'_a \leq e_a$. We also claim that $\mu_{aa} = 0$. Indeed, the definition of $\mu_{aa}$ is
	\[\mu_{aa} = \max\{\gamma_{ak} : e_k < e_a\}.\]
	But if $e_k < e_a$, then by the definition of $e'_a$ we have that $e_k e'_a = 0$, and so $\gamma_{ak} = 0$. So $\mu_{aa} = 0$. This implies that
	\[(1)\mu_{ai} + \gamma_{ai}(0) - (0)\mu_{ai} = \mu_{ai} = 1.\]
	By definition of $\mu_{ai}$, this means that there is some $j$ such that
	\[e'_a \leq e_j < e_i.\]
	If $j = a$, then $e'_a \leq e_a < e_i$, and we are done. So suppose $j \neq a$. Then repeating the above argument with $j$ instead of $i$ yields some $k$ such that
	\[e'_a \leq e_k < e_j < e_i.\]
	Once again, if $k = a$, then $e'_a \leq e_a < \cdots < e_i$, and we are done. Otherwise, we can repeatedly apply the above argument, and at each step we have indices $j, k, \dots, b$ such that
	\[e'_a \leq e_b < \cdots < e_k < e_j < e_i.\]
	Now since there are a finite number of projections, this process must eventually terminate (otherwise we would get an infinite, strictly decreasing chain of projections, which is impossible). So at some point the chain will look like
	\[e'_a \leq e_a < \cdots < e_j < e_i,\]
	which shows that $e_a \leq e_i$ as desired.
\end{proof}

\begin{proof}[Proof of Proposition~\ref{prop:CEL iff}]
	For (1), we first assume that $\{e_1, \dots, e_n\}$ is independent and $e_i e_j = \bigvee_{e_k \leq e_i e_j} e_k$ for any $i, j \in \{1, \dots, n\}$. The independence of $\{e_1, \dots, e_n\}$ is equivalent to each $e'_i$ being non-zero, essentially by definition. Note also that for each $i \in \{1, \dots, n\}$, we have $e'_i \leq e_i$ (again, by definition). Fix distinct $i, j \in \{1, \dots, n\}$. Independence implies that $e_i \neq e_j$. If $e_j < e_i$, then by definition of $e'_i$, we get that $e'_i e_j = 0$, which implies that $e'_i e'_j = e'_i e_j e'_j = 0$. If $e_j$ is not less than $e_i$, then we must have $e_i e_j < e_j$. In this case,
	\[e_i e_j = \bigvee_{e_k \leq e_i e_j} e_k \leq \bigvee_{e_k < e_j} e_k,\]
	and multiplying through by $e'_j$ gives that $e_i e'_j < 0$, and hence $e_i e'_j = 0$. This in turns implies that $e'_i e'_j = 0$. Thus, $e'_1, \dots, e'_n$ are pairwise orthogonal. To see that $e'_1, \dots, e'_n$ linearly span $D$, it suffices to show that $\dim D \leq n$. For this, note that $e_ie_j$ is either $e_i$ or $e_j$, or is the join $\bigvee_{e_k < e_i,e_j} e_k$, which is a linear combination of products of the projections $e_1,\dots,e_n$. If we keep writing products as single projections or joins, then since there are only finitely many $e_1,\dots,e_n$, we eventually get $e_ie_j$ expressed as a linear combination of  these projections. An inductive argument gives that any finite product of the $e_i$'s can be written as a linear combination of the $e_i$'s, and it follows that $\dim D \leq n$.
	
	Now suppose that $\{e'_1, \dots, e'_n\}$ is a family of non-zero, pairwise orthogonal projections linearly spanning $D$. Then each $e_i e_j$, being a projection in $D$, can be expressed as a sum
	\[e_i e_j = \sum_{e'_a \leq e_i e_j} e'_a,\]
	and so Lemma~\ref{lemma:CEL iff lemma} implies that
	\[e_i e_j = \bigvee_{e'_a \leq e_i e_j} e_a,\]
	as required.
	
	For (2), we just note that Lemma~\ref{lemma:CEL iff lemma} gives $\gamma_{ij}=1\iff e_i\le e_j$, and then the arguments in the proof of \cite[Lemma~3.8]{CEL} give the result.
\end{proof}

\subsection{Results on $K$-theory} \label{subsec:CEL generalisation}

We now state and prove Theorem~\ref{thm:mainCELthm}, which is a generalisation of \cite[Corollary~3.14~and~Lemma~3.16]{CEL}.

\begin{notation}
    In the below statement, we denote by $\xi_i$ the standard basis element in $C_0(I)$, and by $e_{i,j}$ the rank-one partial isometry in $\KK(\ell^2(I))$ which sends $\xi_j$ to $\xi_i$. We denote by $u$ the canonical unitary representation of a group into any crossed product induced from an action of the group.
\end{notation}

\begin{thm}\label{thm:mainCELthm}
	Let $G$ be a second countable locally compact group acting continuously on a second countable totally disconnected locally compact space $\Omega$, and denote by $\tau$ the induced action of $G$ on $C_0(\Omega)$.  Let $A$ be a separable $C^*$-algebra, and $\alpha$ an action of $G$ on $A$. Suppose that $\mathcal{V}=\{V_i:i\in I\}$ is a $G$-invariant collection of compact open subsets of $\Omega$ that generates the compact open subsets of $\Omega$, and is independent and finitely aligned. Suppose that $G$ satisfies the Baum--Connes conjecture with coefficients in $A\otimes C_0(I)$ and $A\otimes C_0(\Omega)$. Then there is a homomorphism 
	\[
	\bigoplus_{[i]\in G\backslash I} A\rtimes_{\alpha,r}G_i \to \KK(\ell^2(I))\otimes\big((A\otimes C_0(\Omega))\rtimes_{\alpha\otimes\tau,r}G\big),
	\]
	which, when restricted to the $[i]$ summand, satisfies $au_g\mapsto  e_{i, i}\otimes ((a\otimes 1_{V_i})u_g)$ for all $a\in A,g\in G_i$, where $G_i \leq G$ is the stabiliser subgroup at $i \in I$, under the action of $G$ on $I$ satisfying $g\cdot V_i = V_{g \cdot i}$ for all $g\in G, i\in I$. Moreover, this homomorphism induces an isomorphism 
	\[
	\bigoplus_{[i] \in G\backslash I} K_*(A\rtimes_{\alpha,r}G_i)\cong K_*((A\otimes C_0(\Omega))\rtimes_{\alpha\otimes\tau,r} G) ,
	\]
	which, when restricted to the $[i]$ summand, is the map in $K$-theory induced from the inclusion $A\rtimes_{\alpha,r}G_i\to (A\otimes C_0(\Omega))\rtimes_{\alpha\otimes\tau,r}G,\, au_g\mapsto (a\otimes 1_{V_i})u_g$. When $A=\C$, we get an isomorphism 
	\[
	\bigoplus_{[i] \in G\backslash I} K_*(C^*_r(G_i))\cong K_*(C_0(\Omega)\rtimes_{\tau,r} G).
	\]
\end{thm} 

We denote by $\mu$ the action of $G$ on $C_0(I)$ induced from the action of $G$ on $I$ described in Theorem~\ref{thm:mainCELthm}. Since we are generalising \cite[Corollary~3.14]{CEL}, we first discuss its proof. The proof is broken into two main parts:
\begin{itemize}[leftmargin=5em]
	\item[(Step 1)] the homomorphism $\bigoplus_{[i]\in G\backslash I} A\rtimes_{\alpha,r}G_i \to (A\otimes C_0(I))\rtimes_{\alpha\otimes\mu,r}G$ (called $\eta$ in \cite[Remark~3.15]{CEL}), which induces an isomorphism
	\[
	\bigoplus_{[i]\in G\backslash I} K_*(A\rtimes_{\alpha,r}G_i)\cong K_*((A\otimes C_0(I))\rtimes_{\alpha\otimes\mu,r}G);
	\]
	and
	\item[(Step 2)] the homomorphism $(A\otimes C_0(I))\rtimes_{\alpha\otimes\mu,r}G\to \KK(\ell^2(I))\otimes\big((A\otimes C_0(\Omega))\rtimes_{\alpha\otimes\tau,r}G\big)$ (which, up to isomorphism of the codomain, is the map $\psi$ from \cite[Remark~3.15]{CEL}), which induces an isomorphism 
	\[
	K_*((A\otimes C_0(I))\rtimes_{\alpha\otimes\mu,r}G) \cong K_*((A\otimes C_0(\Omega))\rtimes_{\alpha\otimes\tau,r} G).
	\]
\end{itemize}
The result in (Step 1) applies in our setting, but we require some adjustments to prove (Step 2) with a finitely-aligned $\VV$. 

The proof of (Step 2) in \cite{CEL} proceeds as follows. The authors first use \cite[Proposition~3.3]{CEL} to reduce the problem to actions involving compact subgroups $H$ of $G$. Then in \cite[Lemma~3.6]{CEL} (see also \cite[Remark~3.7]{CEL} for the translation between $H$-equivariant $K$-theory and the $K$-theory of the corresponding crossed products by $H$), the problem is further reduced to considering finite subsets of $\VV$. In particular, for $J \in \FF_H$ (recall that $\FF_H$ is the set of all finite $H$-invariant subsets of $I$), and $C_0(\Omega)_J:=C^*(\{1_{V_j}:j\in J\})\subseteq C_0(\Omega)$, if the maps $K_*(C_0(J)\rtimes H)\to K_*(C_0(\Omega)_J\rtimes H)$ are isomorphisms for every $J$ in some cofinal subset of the set $\FF_H$ ordered by inclusion, then there is an isomorphism 
\[
K_*((A\otimes C_0(I))\rtimes_{\alpha\otimes\mu}H)\cong K_*((A\otimes C_0(\Omega))\rtimes_{\alpha\otimes\tau}H).
\]
Finally, the key step in getting isomorphisms $K_*(C_0(J)\rtimes H)\cong K_*(C_0(\Omega)_J\rtimes H)$ for every compact $H$ and every $J \in \FF_H$, is the following lemma:
\begin{lemma}[Lemma 3.8 of \cite{CEL}] \label{CEL 3.8}
	Let $D$ be a commutative $C^*$-algebra generated by a multiplicatively closed (up to 0) and independent finite family of projections $\{p_i : i \in I\}$. For each $i \in I$ let $p'_i := p_i - \bigvee_{p_j < p_i} p_j$. Then $\{p'_i\}$ is a family of nonzero pairwise orthogonal projections spanning $D$. The transition matrix $\Gamma = (\gamma_{ij})$ determined by the equations $p_j = \sum_{i \in I} \gamma_{ij} p'_i$ is unipotent and therefore invertible over $\Z$.
\end{lemma}
In \cite[Appendix A]{CEL} the authors show how the matrix $\Gamma$ gives rise to a class $x_\Gamma \in KK^H(C_0(J), C_0(\Omega)_J)$, and prove that $\Gamma$ being invertible implies that $x_\Gamma$ is invertible. Remark 3.9 of \cite{CEL} then explains how $x_\Gamma$ being invertible ensures that the maps $K_*(C_0(J)\rtimes H)\to K_*(C_0(\Omega)_J\rtimes H)$ are isomorphisms; this argument relies on the entries of $\Gamma$ being elements of $\{0,1\}$, which is shown in the proof of \cite[Lemma~3.8]{CEL}.

We now state and prove our replacement of (Step 2) for our more general setting. This, when combined with (Step 1), gives Theorem~\ref{thm:mainCELthm}.

\begin{thm}\label{thm:CELmap2}
	Let $G$ be a second countable locally compact group acting continuously on a second countable totally disconnected locally compact space $\Omega$, and denote by $\tau$ the induced action of $G$ on $C_0(\Omega)$.  Let $A$ be a separable $C^*$-algebra, and $\alpha$ an action of $G$ on $A$. Suppose that $\mathcal{V}=\{V_i:i\in I\}$ is a $G$-invariant collection of compact open subsets of $\Omega$ that generates the compact open subsets of $\Omega$, and is independent and finitely aligned. Denote by $\mu$ the action of $G$ on $C_0(I)$ induced from the action of $G$ on $I$ satisfying $g\cdot V_i = V_{g \cdot i}$ for all $g\in G, i\in I$. Suppose that $G$ satisfies the Baum--Connes conjecture with coefficients in $A\otimes C_0(I)$ and $A\otimes C_0(\Omega)$. Then there is a homomorphism 
	\[
	(A\otimes C_0(I))\rtimes_{\alpha\otimes\mu,r}G\to \KK(\ell^2(I))\otimes\big((A\otimes C_0(\Omega))\rtimes_{\alpha\otimes\tau,r}G\big)
	\]
	satisfying $(a\otimes\xi_i)u_g\mapsto e_{i,g^{-1}\cdot i}\otimes ((a\otimes 1_{V_i})u_g)$ for all $a\in A,g\in G, i\in I$, which induces an isomorphism
	\[
	K_*((A\otimes C_0(I))\rtimes_{\alpha\otimes\mu,r} G) \cong K_*((A\otimes C_0(\Omega))\rtimes_{\alpha\otimes\tau,r} G).
	\]
\end{thm}

\begin{proof}
    The above discussion of the proof of \cite[Corollary~3.14]{CEL} means we need to provide: a cofinal subset of $\FF_H$ such that for every $J$ in that subset, the maps $K_*(C_0(J)\rtimes H)\to K_*(C_0(\Omega)_J\rtimes H)$ are isomorphisms; and a suitable replacement for Lemma~\ref{CEL 3.8}. For this we take $\JJ_H$ from Proposition~\ref{prop:cofinality}, and part (2) of Proposition~\ref{prop:CEL iff}.
\end{proof}

This completes the proof of Theorem~\ref{thm:mainCELthm}.

\section{Group actions on multitrees}\label{sec:multitreektheory}

Our main application of Theorem~\ref{thm:mainCELthm} is in studying the $K$-theory of the reduced crossed product of a group acting on the boundary of a certain type of directed graph (a \textit{finitely aligned multitree}). These graphs generalise directed trees and also allow us to study group actions on the boundaries of \textit{undirected} trees, which will be discussed in Section~\ref{sec:undirectedtrees}.

We start with a discussion of multitrees in Section~\ref{subsec:multitrees}, and then in Section~\ref{subsec:multitree K-theory} we explain how Theorem~\ref{thm:mainCELthm} can be applied to this setting. Following this, we give formulae for the $K$-theory of the crossed product in the special cases where the group acts freely on the multitree (Section~\ref{subsec:trivialvertexstabs}) and where the vertex stabilisers of the action are infinite cyclic (Section~\ref{subsec:icvertexstabs}).

\subsection{Multitrees} \label{subsec:multitrees}

In this subsection, we give a definition of multitrees, and prove some useful results about their structure.

\begin{defin}
	A \textit{multitree} is a directed graph $E = (E^0, E^1, r, s)$ such that for any two vertices $v, w \in E^0$, there is at most one (directed) path from $w$ to $v$ (or equivalently, at most one directed path between $v$ and $w$). This condition ensures that there is a partial order on the vertex set $E^0$ defined by $v \leq w$ if and only if there is a directed path from $w$ to $v$. We call a multitree $E$ \textit{finitely aligned} if for any pair of vertices $v, w \in E^0$, the collection $v\vee w$ of minimal common upper bounds of $v$ and $w$, under the partial order described above, is finite.
\end{defin}

Multitrees have also been called \textit{strongly unambiguous graphs} or \textit{mangroves} in \cite{mangroves}.
\vspace{1ex}
\begin{eg}[Example of a multitree]~
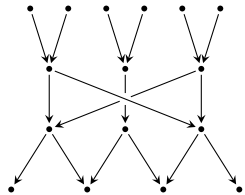
    \[\begin{tikzpicture}
        \foreach \i in {0,1,2,3} {
            \node[vertex] (a\i) at (\i,0) {};
        }
        
        \foreach \i in {0,1,2} {
            \node[vertex] (b\i) at (\i+0.5,0.8) {};
            \node[vertex] (c\i) at (\i+0.5,1.6) {};
        }
        
        \path (b1) -- (c1) node[midway,vertex,white] (m) {};
        
        \foreach \i in {0,1,2,3,4,5}
        \node[vertex] (d\i) at (\i/2+0.25,2.4) {};
        
        \draw[-stealth]
        (b0) edge (a0)
        (b0) edge (a1)
        (b1) edge (a1)
        (b1) edge (a2)
        (b2) edge (a2)
        (b2) edge (a3)
        
        (c0) edge (b0)
        (c0) edge (b2)
        (c1) edge[-] (m)
        (m) edge (b1)
        (c2) edge[-] (m)
        (m) edge (b0)
        (c2) edge (b2)
        
        (d0) edge (c0)
        (d1) edge (c0)
        (d2) edge (c1)
        (d3) edge (c1)
        (d4) edge (c2)
        (d5) edge (c2);
    \end{tikzpicture}\]
\end{eg}

In this paper, we restrict our attention to multitrees $E$ which are row-finite, finitely aligned and have no sources. We can define the boundary $\partial E$ of a multitree $E$ in the same way as for directed trees (as described in Section~\ref{subsec:boundaries}), and it will also be a totally disconnected locally compact Hausdorff space. Given an action of a countable discrete group $G$ on $E$, we wish to apply Theorem~\ref{thm:mainCELthm} to calculate the $K$-theory of the reduced crossed product induced from the action of $G$ on the boundary $\partial E$. However, this boundary will not have a suitable collection of compact open subsets, so in order to apply the theorem, we will need to enlarge the space with the vertices $E^0$.

\begin{defin} \label{def:Omega}
	Suppose that $E$ is a row-finite, finitely aligned multitree without sources, and consider the space $E^0\cup\partial E$. We define the cylinder set of a vertex $v \in E^0$ to be $Z(v) := \{w \in E^0 : vE^*w \neq \emptyset\} \cup Z_{\partial}(v)$. The collection $\{\{v\} : v \in E^0\} \cup \{Z(v) : v \in E^0\}$ forms a basis for a topology of compact open sets turning $E^0 \cup \partial E$ into a totally disconnected locally compact Hausdorff space.
\end{defin}

The following proposition will be useful for Section~\ref{subsec:multitree K-theory}.

\begin{prop} \label{prop:multitree cylinder intersection}
	Let $E$ be a row-finite, finitely aligned multitree with no sources. For any two vertices $v, w \in E^0$, we have that
	\[Z(v) \cap Z(w) = \bigsqcup_{u\in v\vee w} Z(u).\]
\end{prop}

We will use the following lemma in the proof of Proposition~\ref{prop:multitree cylinder intersection}.

\begin{lemma} \label{lemma:multitree cylinder sets}
	Let $E$ be a row-finite, finitely aligned multitree with no sources, and let $v, w \in E^0$. The following are equivalent:
	\begin{enumerate}[(i)]
		\item $v \leq w$;
		\item $Z(w) \subseteq Z(v)$;
		\item $w \in Z(v)$.
	\end{enumerate}
\end{lemma}

\begin{proof}
	We show that (i) $\implies$ (ii) $\implies$ (iii) $\implies$ (i).
	
	First suppose (i), so $v \leq w$. This means that there is a path $\lambda \in vE^*w$. Now take $x \in Z(w)$. We treat the two cases $x \in E^0$ and $x \in \partial E$ separately. If $x \in E^0$, then by definition of $Z(w)$, there is a path $\mu \in wE^*x$. So $\lambda\mu \in vE^*x$, which implies that $x \in Z(v)$. If $x \in \partial E$, then $x = [\nu]$ for some $\nu \in wE^\infty$. So $\lambda\nu \in vE^\infty$, and so $[\lambda\nu] = [\nu] = x$ is again in $Z(v)$. This shows that $Z(w) \subseteq Z(v)$, so (ii) holds.
	
	The implication (ii) $\implies$ (iii) follows immediately from the fact that $w \in Z(w)$. Finally, suppose (iii), so $w \in Z(v)$. Since $w \in E^0$, this means that there is a path $\lambda \in vE^*w$, which exactly means that $v \leq w$, which is (i).
\end{proof}

\begin{proof}[Proof of Proposition~\ref{prop:multitree cylinder intersection}]
	Fix $v, w \in E^0$. First we show that
	\begin{equation} \label{eq:multitree cylinder intersection nondisjoint}
		Z(v) \cap Z(w) = \bigcup_{u\in v\vee w} Z(u).
	\end{equation}
	For the backwards inclusion, fix a minimal common upper bound $u$ of $v$ and $w$. Since both $v \leq u$ and $w \leq u$, Lemma~\ref{lemma:multitree cylinder sets} implies that $Z(u) \subseteq Z(v) \cap Z(w)$. So the right-hand side of Equation~\eqref{eq:multitree cylinder intersection nondisjoint} is also contained in $Z(v) \cap Z(w)$.
	
	For the forwards inclusion, take $x \in Z(v) \cap Z(w)$. We treat the two cases $x \in E^0$ and $x \in \partial E$ separately. First suppose that $x \in E^0$. Then by Lemma~\ref{lemma:multitree cylinder sets}, we have that $v \leq x$ and $w \leq x$. So $x$ is a common upper bound of $v$ and $w$, which means that there must be some minimal common upper bound $u$ of $v$ and $w$ such that $u \leq x$. Then Lemma~\ref{lemma:multitree cylinder sets} gives that $x \in Z(u)$. Now suppose that $x = [\lambda] \in \partial E$, for some $\lambda \in E^\infty$. This means that there are infinite paths $\lambda_v \in vE^\infty$ and $\lambda_w \in wE^\infty$ such that $\lambda$, $\lambda_v$ and $\lambda_w$ are shift-tail equivalent to each other. So there is an infinite path $\mu \in E^\infty$ and finite paths $\nu_v \in vE^*r(\mu)$ and $\nu_w \in wE^*r(\mu)$ such that $\lambda_v = \nu_v \mu$ and $\lambda_w = \nu_w \mu$. Then the vertex $r(\mu)$ is a common upper bound of $v$ and $w$, and by the same argument as in the first case, there is some minimal common upper bound $u$ of $v$ and $w$ such that $u \leq r(\mu)$. Let $\nu_u$ be the path from $r(\mu)$ to $u$. Then $\nu_u\mu \in uE^\infty$ is shift-tail equivalent to $\lambda$ and so $x = [\lambda] = [\nu_u\mu]$ lies in $Z(u)$. This proves the forwards inclusion and hence we have shown Equation~\eqref{eq:multitree cylinder intersection nondisjoint}.
	
	It remains to prove that the union $\cup_u Z(u)$ is disjoint. Suppose for a contradiction that there are two distinct minimal common upper bounds $u_1$ and $u_2$ of $v$ and $w$ such that $Z(u_1) \cap Z(u_2) \neq \emptyset$. Equation~\eqref{eq:multitree cylinder intersection nondisjoint}, with $u_1$ and $u_2$ in place of $v$ and $w$ respectively, implies that $u_1$ and $u_2$ have a (minimal) common upper bound $u' \in E^0$. So there are paths $\lambda_1 \in u_1E^*u'$ and $\lambda_2 \in u_2E^*u'$. Also, since $u_1$ and $u_2$ are both upper bounds of $v$, there are paths $\mu_1 \in vE^*u_1$ and $\mu_2 \in vE^*u_2$. So then we have two paths $\mu_1\lambda_1$ and $\mu_2\lambda_2$ which both go from $u'$ to $v$. Moreover, the paths $\mu_1\lambda_1$ and $\mu_2\lambda_2$ must be distinct, since otherwise $u_1$ and $u_2$ would lie on the same directed path and would therefore be comparable, contradicting the fact that $u_1$ and $u_2$ are distinct minimal common upper bounds of $u$. But since there are two distinct paths from $u'$ to $v$, this contradicts the fact that $E$ is a multitree, and so the intersection $Z(u_1) \cap Z(u_2)$ must be empty. This shows that the union $\cup_u Z(u)$ is disjoint.
\end{proof}

\subsection{$K$-theory} \label{subsec:multitree K-theory}

Suppose a group $G$ acts on a row-finite, finitely aligned multitree $E$ with no sources. Our main result of this section is Theorem~\ref{thm:multitree sequence}, which describes the $K$-theory of the reduced crossed product of $G$ acting on the boundary of $E$. Before we can state the theorem, we introduce some notation.

\begin{notation}\label{not:thetaetc}
  In this subsection, we will always write $\Omega$ for the topological space $E^0 \cup \partial E$ from Definition~\ref{def:Omega}. We will write $\tau$ for the induced action of $G$ on $C_0(\Omega)$, and we will write $\tau_0$ and $\tau_\partial$ for the restriction of $\tau$ to $C_0(E^0)$ and $C_0(\partial E)$ respectively. For each $v \in E^0$, we write $\xi_v \in C_0(E^0)$ for the indicator function of the singleton set $\{v\}$, and we write $1_{Z(v)} \in C_0(\Omega)$ for the indicator function of the set $Z(v) \subseteq \Omega$.

  For each $e \in E^1$, we will fix some transversal $\Sigma_e$ for  $G_{r(e)}/G_e$ containing the identity element $1_{G_{r(e)}}$ (which is just $1_G$). Note that $\Sigma_e$ has a natural left action of $G_{r(e)}$ defined by setting $g \cdot h$ to be the representative in $\Sigma_e$ of the coset $ghG_e$, for $g \in G_{r(e)}$, $h \in \Sigma_e$. For $h,k \in \Sigma_e$, we denote by $e_{h,k} \in \KK(\ell^2(\Sigma_e))$ the rank-one partial isometry sending the standard basis element $\xi_k \in \ell^2(\Sigma_e)$ at $k$ to the standard basis element $\xi_h$ at $h$. We denote by $u$ the canonical unitary representation of a group into its reduced group $C^*$-algebra. Finally, we write $\theta_e$ for the map defined by
	\begin{align*}
		\theta_e \colon C^*_r(G_{r(e)}) &\to \mathcal{K}(\ell^2(\Sigma_e)) \otimes C^*_r(G_{s(e)}) \\
		u_g &\mapsto \sum_{h \in \Sigma_e} e_{h,g^{-1}\cdot h} \otimes u_{h^{-1} g (g^{-1} \cdot h)}.
	\end{align*}
\end{notation}

Note that for any $g \in G_{r(e)}$ and $h \in \Sigma_e$, the elements $g^{-1}h$ and $g^{-1} \cdot h$ are in the same coset in $G_{r(e)}/G_e$, so the element $h^{-1}g(g^{-1}\cdot h)$ is in $G_e$ and therefore also in $G_{s(e)}$. Also, recall that the row-finiteness of $E$ implies that $\Sigma_e$ is finite, so the sum in the above formula for $\theta_e$ is well-defined.

We are primarily interested in the induced map $(\theta_e)_*$ on $K$-theory; the following lemma shows that this map is (essentially) independent of the choice of $e \in E^1$ in a fixed edge orbit. Note that the lemma also implies that $(\theta_e)_*$ is independent of the choice of transversal $\Sigma_e$, by taking $e' = e$ and $g$ to be the identity element of $G$.

\begin{lemma} \label{lem:theta_e diagram}
  Let $G$ be a countable discrete group acting on a row-finite, finitely aligned multitree $E$ with no sources. Let $e, e' \in E^1$ be edges in the same orbit, and let $g \in G$ be an element such that $e' = g \cdot e$. Write $\Ad u_g \colon C^*_r(G_{r(e)}) \to C^*_r(G_{r(e')})$ for the *-homomorphism defined by $u_h \mapsto u_{ghg^{-1}}$ for $h \in G_{r(e)}$. We have the following commutative diagram on $K$-theory
  \[\begin{tikzcd}
    K_*(C^*_r(G_{r(e)})) \arrow[r,"(\theta_e)_*"] \arrow[d,"(\Ad u_g)_*"] & K_*(C^*_r(G_{s(e)})) \arrow[d,"(\Ad u_g)_*"] \\
    K_*(C^*_r(G_{r(e')})) \arrow[r,"(\theta_{e'})_*"] & K_*(C^*_r(G_{s(e')})),
  \end{tikzcd}\]
  where the downwards arrows are isomorphisms.
\end{lemma}

\begin{proof}
  The claim about the isomorphisms follows from the fact that $G_{r(e')} = g G_{r(e)} g^{-1}$ and similarly $G_{s(e')} = g G_{s(e)} g^{-1}$. Note also that $G_{e'} = g G_e g^{-1}$. For each $h \in \Sigma_e$, write $h' \in \Sigma_{e'}$ for the representative in $\Sigma_{e'}$ of the coset $ghg^{-1}G_{e'}$, and write $h'' \in G_{s(e')}$ for the element $(h')^{-1}ghg^{-1}$ (note that $h''$ is in $G_{e'}$ and therefore $G_{s(e')}$ because both $h'$ and $ghg^{-1}$ are in the same coset in $G_{r(e')}/G_{e'}$).
  
  First we claim that for any $k \in G_{r(e)}$ and any $h \in \Sigma_e$, we have that $(gkg^{-1})^{-1} \cdot h' = (k^{-1} \cdot h)'$. To see this, first note that $(gkg^{-1})^{-1} \cdot h'$ is the representative in $\Sigma_{e'}$ of the coset
  \[gk^{-1}g^{-1}h'G_{e'} = gk^{-1}g^{-1}ghg^{-1}G_{e'} = gk^{-1}hg^{-1}G_{e'}.\]
  On the other hand, $(k^{-1} \cdot h)'$ is the representative in $\Sigma_{e'}$ of the coset
  \[g(k^{-1} \cdot h)g^{-1}G_{e'} = g(k^{-1} \cdot h)G_e g^{-1} = gk^{-1}hG_e g^{-1} = gk^{-1}hg^{-1}G_{e'}.\]
  So $(gkg^{-1})^{-1} \cdot h'$ and $(k^{-1} \cdot h)'$ are both representatives in $\Sigma_{e'}$ of the same coset, and hence they must be equal, as claimed.
  
  Now note that for $h \in \Sigma_e$, we have $h' = ghg^{-1}(h'')^{-1}$. So for any $k \in G_{r(e)}$, we have
  \begin{align*}
    (h')^{-1}(gkg^{-1})((gkg^{-1})^{-1} \cdot h') &= (h')^{-1}(gkg^{-1})(k^{-1} \cdot h)' \\
    &= (h''gh^{-1}g^{-1})(gkg^{-1})[g(k^{-1}\cdot h)g^{-1}((k^{-1} \cdot h)'')^{-1}] \\
    &= h''gh^{-1}k(k^{-1}\cdot h)g^{-1}[(k^{-1}\cdot h)'']^{-1}.
  \end{align*}
  So, writing $\varphi \colon \KK(\ell^2(\Sigma_e)) \to \KK(\ell^2(\Sigma_{e'}))$ for the isomorphism induced by the map $h \mapsto h'$, and writing $u \in \KK(\ell^2(\Sigma_{e'})) \otimes C^*_r(G_{s(e')})$ for the unitary defined by
  \[u = \sum_{h \in \Sigma_e} e_{h',h'} \otimes u_{h''},\]
  we have
  \begin{align*}
    (\theta_{e'} \circ \Ad u_g)(u_k) &= \theta_{e'}(u_{gkg^{-1}}) \\
    &= \sum_{h' \in \Sigma_{e'}} e_{h',(gkg^{-1})^{-1} \cdot h'} \otimes u_{(h')^{-1}(gkg^{-1})((gkg^{-1})^{-1} \cdot h')} \\
    &= \sum_{h \in \Sigma_e} e_{h',(k^{-1} \cdot h)'} \otimes u_{h''gh^{-1}k(k^{-1}\cdot h)g^{-1}[(k^{-1}\cdot h)'']^{-1}} \\
    &= u \left( \sum_{h \in \Sigma_e} e_{h',(k^{-1} \cdot h)'} \otimes u_{g(h^{-1}k(k^{-1}\cdot h))g^{-1}} \right) u^* \\
    &= (\Ad u)((\varphi \otimes \Ad u_g)(\theta_e(u_k))).
  \end{align*}
  Since $\theta_{e'} \circ \Ad u_g$ and $\Ad u \circ (\varphi \otimes \Ad u_g) \circ \theta_e$ agree on a generating set of a dense subspace of $C^*_r(G_{r(e)})$, they must coincide. This means that their induced maps on $K$-theory must also be the same, and noting that $\Ad u$ induces the identity map in $K$-theory as it is an inner automorphism, we get the commutative diagram in the statement of the lemma.
\end{proof}

From now on, we implicitly use Lemma~\ref{lem:theta_e diagram} to freely identify the groups $K_*(C^*_r(G_v))$ for vertices $v \in E^0$ in the same vertex orbit, as well as to identify the maps $(\theta_e)_*$ for edges $e \in E^1$ in the same edge orbit. When convenient, we will also think of $(\theta_e)_*$ as a map $\oplus_{[v]} K_*(C^*_r(G_v)) \to \oplus_{[w]} K_*(C^*_r(G_w))$ by defining $(\theta_e)_*$ to be the zero map on all summands of the domain apart from the $[r(e)]$ summand, and extending the codomain in the natural way.

\begin{thm} \label{thm:multitree sequence}
	Let $G$ be a countable discrete group acting on a row-finite, finitely aligned multitree $E$ with no sources. Suppose that the action $G \curvearrowright \Omega$ is amenable. Then for $\alpha_i := \sum_{[e] \in G\backslash E^1} (\theta_e)_{*,i}$, $i=0,1$, we have the following six-term exact sequence
	\[\begin{tikzcd}
		\bigoplus_{[v]\in G\backslash E^0} K_0(C^*_r(G_v))
		\arrow[r,"\id-\alpha_0"] &
		\bigoplus_{[v]\in G\backslash E^0} K_0(C^*_r(G_v)) \arrow[r] &
		K_0(C_0(\partial E) \rtimes_{\tau_\partial,r} G) \arrow[d] \\
		K_1(C_0(\partial E) \rtimes_{\tau_\partial,r} G) \arrow[u] &
		\bigoplus_{[v]\in G\backslash E^0} K_1(C^*_r(G_v)) \arrow[l] &
		\bigoplus_{[v]\in G\backslash E^0} K_1(C^*_r(G_v)),
		\arrow[l,"\id-\alpha_1"]
	\end{tikzcd}\]
\end{thm}

\begin{remark}\label{rem:amenabilityassumptions}
  We note that the action $G \curvearrowright \Omega$ is amenable if and only if $G_v$ is amenable for all $v \in E^0$ and the action $G \curvearrowright \partial E$ is amenable. This is because:
  \begin{itemize}
    \item $G \curvearrowright \Omega$ is amenable if and only if both induced actions $G \curvearrowright E^0$ and $G \curvearrowright \partial E$ are amenable, by Proposition~3.23 and Theorem~5.15 of \cite{BEW} taken together; and
    \item $G \curvearrowright E^0$ is amenable if and only if $G_v$ is amenable for all $v \in E^0$: per \cite[Remarque~4.10]{AD}, the amenability of $G \curvearrowright E^0$ is equivalent to the nuclearity of the crossed product $C_0(E^0) \rtimes G \cong \bigoplus_{[v] \in G\backslash E^0} \KK(\ell^2(G \cdot v)) \otimes C^*(G_v)$, and this is nuclear if and only if $G_v$ is amenable for all $v \in E^0$.
  \end{itemize}
\end{remark}

The six-term exact sequence in the theorem is essentially the one coming from the short exact sequence of $C^*$-algebras
\[0 \to C_0(E^0) \rtimes_{\tau_0,r} G \xrightarrow{\iota} C_0(\Omega) \rtimes_{\tau,r} G \to C_0(\partial E) \rtimes_{\tau_\partial,r} G \to 0,\]
where $\iota\colon C_0(E^0) \rtimes_{\tau_0,r} G \to C_0(\Omega) \rtimes_{\tau,r} G$ is the canonical inclusion. The following two propositions together yield the sequence in the theorem.

\begin{prop} \label{prop:multitree Omega K-theory}
	Let $G$ be a countable discrete group acting on a row-finite, finitely aligned multitree $E$ with no sources. Suppose that the action $G \curvearrowright \Omega$ is amenable. We have the following isomorphisms in $K$-theory:
	\[\bigoplus_{[v]\in G\backslash E^0} K_*(C^*_r(G_v)) \cong K_*(C_0(E^0) \rtimes_{\tau_0,r} G) \quad \text{and} \quad \bigoplus_{[v]\in G\backslash E^0} K_*(C^*_r(G_v)) \cong K_*(C_0(\Omega) \rtimes_{\tau,r} G).\]
	Moreover, on each direct summand, the isomorphisms are induced by the *-homomorphisms
	\begin{align*}
		l_v \colon C^*_r(G_v) &\to C_0(E^0) \rtimes_{\tau_0,r} G
		& \text{and} &&
		l_{Z(v)}\colon C^*_r(G_v) &\to C_0(\Omega) \rtimes_{\tau,r} G\\
		u_g &\mapsto \xi_v u_g
		&&&
		u_g &\mapsto 1_{Z(v)} u_g
	\end{align*}
	respectively.
\end{prop}

\begin{proof}
	The claims about the first isomorphism follow from \cite[Remarks~3.13 and 3.15]{CEL}.
	
	For the second isomorphism, we claim that the collection of compact open sets $\mathcal{V} = \{Z(v) : v \in E^0\}$ in $\Omega$ satisfies the conditions of Theorem~\ref{thm:mainCELthm}. The collection $\mathcal{V}$ is $G$-invariant since the action of $G$ on $\Omega$ satisfies $g \cdot Z(v) = Z(g \cdot v)$ for all $g \in G$ and $v \in E^0$. To see that $\mathcal{V}$ generates the collection of all compact open subsets of $\Omega$, it is enough to show that every basis set of $\Omega$ can be written in terms of the sets in $\mathcal{V}$ using finite unions, finite intersections and set differences; but for each $v \in E^0$, we have $Z(v) \in \mathcal{V}$ and $\{v\} = Z(v) \setminus \cup_{e \in vE^1} Z(s(e))$, so this is true. To see that $\mathcal{V}$ is independent, note that Lemma~\ref{lemma:multitree cylinder sets} implies that for any vertex $v \in E^0$, the only subset of $Z(v)$ in $\VV$ containing $v$ is $Z(v)$ itself. Finally, finite alignment follows from Proposition~\ref{prop:multitree cylinder intersection} and the assumption that $E$ itself is finitely aligned.
	
	In order to apply Theorem~\ref{thm:mainCELthm}, we also need the group $G$ to satisfy the Baum--Connes conjecture with coefficients in $C_0(E^0)$ and $C_0(\Omega)$. By \cite[Remark~3.13]{CEL}, it is enough that all vertex stabilisers are amenable and that the transformation groupoid $\Omega \rtimes G$ is amenable. But we know from Remark~\ref{rem:amenabilityassumptions} that each vertex stabiliser is amenable, and the amenability of $G \curvearrowright \Omega$ implies that $\Omega \rtimes G$ is amenable by \cite[Examples~2.2.14(2)]{ADR}. Hence the conditions of Theorem~\ref{thm:mainCELthm} are satisfied, and the claims about the second isomorphism follow directly from that theorem.
\end{proof}

\begin{prop} \label{prop:multitree iota_* diagram}
	Let $G$ be a countable discrete group acting on a row-finite, finitely aligned multitree $E$ with no sources. Write $\alpha_i := \sum_{[e] \in G\backslash E^1} (\theta_e)_{*,i}$ for $i = 0, 1$, and write $\iota$ for the canonical inclusion $C_0(E^0) \rtimes_{\tau_0,r} G \to C_0(\Omega) \rtimes_{\tau,r} G$. We have the following commutative diagram for $i = 0,1$:
	\[\begin{tikzcd}
		\bigoplus_{[v] \in G\backslash E^0} K_i(C^*_r(G_v))
		\arrow[r,"\id - \alpha_i"]
		\arrow[d, "\bigoplus_{[v]}(l_v)_{*,i}"] &
		\bigoplus_{[w] \in G\backslash E^0} K_i(C^*_r(G_w))
		\arrow[d,"\bigoplus_{[w]} (l_{Z(w)})_{*,i}"] \\
		K_i(C_0(E^0) \rtimes_{\tau_0,r} G) \arrow[r,"\iota_{*,i}"] &
		K_i(C_0(\Omega) \rtimes_{\tau,r} G).
	\end{tikzcd}\]
\end{prop}

To prove Proposition~\ref{prop:multitree iota_* diagram}, we need some notation and two lemmas.

\begin{notation}
	For each $e \in E^1$, write
	\[1_{[e]} := \sum_{e' \in G_{r(e)} \cdot e} 1_{Z(s(e'))} \in C_0(\Omega)\]
	and write $l_{[e]}$ for the map
	\[C^*_r(G_{r(e)}) \to C_0(\Omega) \rtimes_{\tau,r} G, \quad u_g \mapsto 1_{[e]}u_g.\]
	Write
	\[C_{[e]} := C^*(\{1_{Z(s(e'))} : e' \in G_{r(e)} \cdot e\}) \rtimes_{\tau|,r} G_{r(e)} \subseteq C_0(\Omega) \rtimes_{\tau,r} G.\]
\end{notation}

\begin{lemma}\label{lem:selfadjointunitary}
	Let $G$ be a countable discrete group acting on a row-finite, finitely aligned multitree $E$ with no sources, and fix $e \in E^1$. The element
	\[u_{[e]} = \sum_{h,k \in \Sigma_e} e_{h,k} \otimes u_{k h^{-1}}1_{Z(s(h\cdot e))} \in \mathcal{K}(\ell^2(\Sigma_e)) \otimes C_{[e]}\]
	is a self-adjoint unitary in $\mathcal{K}(\ell^2(\Sigma_e)) \otimes C_{[e]}$.
\end{lemma}

\begin{proof}
	To see that $u_{[e]}$ is self-adjoint, note that
	\begin{align*}
		u_{[e]}^*& = \sum_{h,k \in \Sigma_e} e_{k,h} \otimes 1_{Z(s(h\cdot e))} u^*_{k h^{-1}}\\ 
		&= \sum_{h,k \in \Sigma_e} e_{k,h} \otimes u^*_{k h^{-1}} 1_{Z(s(k\cdot e))}\\ 
		&= \sum_{h,k \in \Sigma_e} e_{k,h} \otimes u_{h k^{-1}} 1_{Z(s(k\cdot e))}\\ 
		&= u_{[e]}.
	\end{align*}
	To see that $u_{[e]}$ is unitary, first note that $C_{[e]}$ is unital with identity element $1_{[e]}$. Now
	\begin{align*}
		u_{[e]} u_{[e]}^* &= \left(\sum_{h,k \in \Sigma_e} e_{h,k} \otimes u_{k h^{-1}} 1_{Z(s(h\cdot e))}\right)\left(\sum_{k,l \in \Sigma_e} e_{k,l} \otimes 1_{Z(s(l\cdot e))} u^*_{k l^{-1}}\right) \\
		&= \sum_{h,l \in \Sigma_e} e_{h,l} \otimes \left( \sum_{k \in \Sigma_e} u_{k h^{-1}} 1_{Z(s(h\cdot e))}1_{Z(s(l\cdot e))} u^*_{k l^{-1}} \right).
	\end{align*}
	Note that the projections $\{1_{Z(s(h\cdot e))} : h \in \Sigma_e\}$ are pairwise orthogonal. So continuing the calculation, we use covariance to get
	\begin{align*}
		u_{[e]} u_{[e]}^* &= \sum_{h \in \Sigma_e} e_{h,h} \otimes \left( \sum_{k \in \Sigma_e} u_{k h^{-1}} 1_{Z(s(h\cdot e))} u^*_{k h^{-1}} \right) \\
		&= \sum_{h \in \Sigma_e} e_{h,h} \otimes \left( \sum_{k \in \Sigma_e} 1_{Z(s(k\cdot e))} u_{k h^{-1}} u^*_{k h^{-1}} \right) \\
		&= \sum_{h \in \Sigma_e} e_{h,h} \otimes \left( \sum_{k \in \Sigma_e} 1_{Z(s(k\cdot e))} \right) \\
		&= \sum_{h \in \Sigma_e} e_{h,h} \otimes 1_{[e]} \\
		&= 1_{\mathcal{K}(\ell^2(\Sigma_e)) \otimes C_{[e]}}.
	\end{align*}
	Since $u_{[e]}$ is self-adjoint, we also have that $u_{[e]}^* u_{[e]} = 1_{\mathcal{K}(\ell^2(\Sigma_e)) \otimes C_{[e]}}$.
\end{proof}

\begin{lemma} \label{lem:multitree theta_e}
	Let $G$ be a countable discrete group acting on a row-finite, finitely aligned multitree $E$ with no sources. For any $e \in E^1$, we have
	\[(l_{Z(s(e))})_*\circ(\theta_e)_* = (l_{[e]})_*.\]
\end{lemma}

\begin{proof}
	Fix $e \in E^1$. Write $l'_{[e]}\colon C^*_r(G_{r(e)}) \to C_{[e]}$ for the map defined by $u_g \mapsto 1_{[e]}u_g$ (this is just $l_{[e]}$ with codomain restricted to $C_{[e]}$), and write $\iota_{[e]}$ for the inclusion $C_{[e]} \into C_0(\Omega) \rtimes_r G$. We need to show that
	\[(l_{Z(s(e))})_*\circ(\theta_e)_* = (\iota_{[e]})_* \circ (l'_{[e]})_*.\]
	Write $\varphi_e$ for the map
	\[\varphi_e \colon C_{[e]} \to \mathcal{K}(\ell^2(\Sigma_e)) \otimes C_{[e]}, \qquad a \mapsto u_{[e]}(e_{1,1} \otimes a)u_{[e]}^*,\]
	where $u_{[e]}$ is the self-adjoint unitary from Lemma~\ref{lem:selfadjointunitary} and where $1 \in \Sigma_e$ denotes the identity of $G_{r(e)}$. Since $\varphi_e$ is the composition of $a \mapsto e_{1,1} \otimes a$ and an inner automorphism, both of which induce the identity map on $K$-theory, the map $\varphi_e$ itself induces the identity map
	\[(\varphi_e)_* = \id\colon K_*(C_{[e]}) \to K_*(C_{[e]}).\]
	So it is enough to show that
    \begin{equation} \label{eq:induced maps on K-theory}
        (l_{Z(s(e))})_*\circ(\theta_e)_* = (\iota_{[e]})_* \circ (\varphi_e)_* \circ (l'_{[e]})_*.
    \end{equation}
	To do this, we show that the following diagram is commutative; Equation~\eqref{eq:induced maps on K-theory} then follows from the commutativity of the induced diagram on $K$-theory:
	\begin{equation} \label{eq:theta_e diagram}
		\begin{tikzcd}
			C^*_r(G_{r(e)}) \arrow[rr,"\theta_e"] \arrow[d, "l'_{[e]}"] & & \mathcal{K}(\ell^2(\Sigma_e)) \otimes C^*_r(G_{s(e)}) \arrow[d,"\id \otimes l_{Z(s(e))}"] \\
			C_{[e]} \arrow[r,"\varphi_e"] & \mathcal{K}(\ell^2(\Sigma_e)) \otimes C_{[e]} \arrow[r,"\id \otimes \iota_{[e]}"]& \mathcal{K}(\ell^2(\Sigma_e)) \otimes (C_0(\Omega) \rtimes_r G).
		\end{tikzcd}
	\end{equation}
	Consider the composition $(\id \otimes \iota_{[e]}) \circ \varphi_e \circ l'_{[e]}\colon C^*_r(G_{r(e)}) \to \mathcal{K}(\ell^2(\Sigma_e)) \otimes (C_0(\Omega) \rtimes_r G)$. For $g \in G_{r(e)}$, we have
	\begin{align*}
		((\id \otimes \iota_{[e]}) \circ \varphi_e \circ l'_{[e]})(u_g) &= ((\id \otimes \iota_{[e]}) \circ \varphi_e)(1_{[e]}u_g) \\
		&= (\id \otimes \iota_{[e]})(u_{[e]} (e_{1,1} \otimes 1_{[e]}u_g) u_{[e]}^*) \\
		&= \sum_{h,k \in \Sigma_e} e_{h,k} \otimes 1_{Z(s(e))} u_{h^{-1}gk} 1_{Z(s(e))} \\
		&= \sum_{h,k \in \Sigma_e} e_{h,k} \otimes 1_{Z(s(e))} 1_{Z(s(h^{-1}gk\cdot e))} u_{h^{-1}gk}.
	\end{align*}
	Note that $1_{Z(s(e))} 1_{Z(s(h^{-1}gk\cdot e))}$ is non-zero if and only if $h^{-1}gk \in G_e$, in which case it is equal to $1_{Z(s(e))}$. But $h^{-1}gk \in G_e$ if and only if $k = g^{-1} \cdot h$ as elements of $\Sigma_e$, so continuing the calculation, we have
	\begin{align*}
		((\id \otimes \iota_{[e]}) \circ \varphi_e \circ l'_{[e]})(u_g) &= \sum_{h \in \Sigma_e} e_{h, g^{-1}\cdot h} \otimes 1_{Z(s(e))} u_{h^{-1}g(g^{-1}\cdot h)} \\
		&= ((\id \otimes l_{Z(s(e))}) \circ \theta_e)(u_g).
	\end{align*}
	Since the maps $(\id \otimes \iota_{[e]}) \circ \varphi_e \circ l'_{[e]}$ and $(\id \otimes l_{Z(s(e))}) \circ \theta_e$ agree on a generating set of a dense sub-$C^*$-algebra of $C^*_r(G_v)$, they are equal, which proves that the diagram Equation~\eqref{eq:theta_e diagram} is commutative.
\end{proof}

We are now ready to prove Proposition~\ref{prop:multitree iota_* diagram}, and then complete the proof of Theorem~\ref{thm:multitree sequence}.

\begin{proof}[Proof of Proposition~\ref{prop:multitree iota_* diagram}]
	Fix $v \in E^0$. Note that
	\[1_{Z(v)} = \xi_v + \sum_{[e] \in G_v\backslash vE^1} 1_{[e]},\]
    so the *-homomorphisms $\iota \circ l_v$ and $l_{[e]}$, $[e] \in G_v\backslash vE^1$ from $C^*_r(G_v) \to C_0(\Omega) \rtimes_{\tau,r} G$ are mutually orthogonal. Thus the induced homomorphisms in $K$-theory satisfy
	\[\iota_* \circ (l_v)_* = (\iota \circ l_v)_* = (l_{Z(v)})_* - \sum_{[e] \in G_v\backslash vE^1} (l_{[e]})_*.\]
	Now, Proposition~\ref{lem:multitree theta_e} says that for all $e \in vE^1$, we have
	\[(l_{Z(s(e))})_*\circ(\theta_e)_* = (l_{[e]})_*.\]
	Summing over a set of representatives of $G_v\backslash vE^1$, we get that
	\[\left(\bigoplus_{[w]\in G\backslash E^0} (l_{Z(w)})_*\right)\circ\left(\sum_{[e] \in G_v\backslash vE^1}(\theta_e)_*\right) = \sum_{[e] \in G_v\backslash vE^1}(l_{[e]})_*.\]
	Note also that, if we write $\iota_v$ for the canonical inclusion of $K_*(C^*_r(G_v))$ into the direct sum $\oplus_{[w]} K_*(C^*_r(G_w))$, we trivially have that
	\[\left(\bigoplus_{[w]\in G\backslash E^0} (l_{Z(w)})_*\right) \circ \iota_v = (l_{Z(v)})_*.\]
	Putting the above two equations together gives that
	\[\left(\bigoplus_{[w]\in G\backslash E^0} (l_{Z(w)})_*\right)\circ\left(\iota_v - \sum_{[e] \in G_v\backslash vE^1}(\theta_e)_*\right) = (l_{Z(v)})_* - \sum_{[e] \in G_v\backslash vE^1}(l_{[e]})_* = \iota_* \circ (l_v)_*.\]
	Note that $\oplus_{[v] \in G\backslash E^0} \iota_v$ is the identity map on $\oplus_{[v] \in G\backslash E^0} K_*(C^*_r(G_v))$, so taking a direct sum of the above equation over $[v] \in G\backslash E^0$ (now allowing $v$ to vary) gives
	\[\left(\bigoplus_{[w]\in G\backslash E^0} (l_{Z(w)})_*\right)\circ\left(\id - \bigoplus_{[v] \in G\backslash E^0}\sum_{[e] \in G_v\backslash vE^1}(\theta_e)_*\right) = \iota_* \circ \left(\bigoplus_{[v] \in G\backslash E^0}(l_v)_*\right).\]
	Now all that remains to be shown is that, for $i=0,1$,
	\[\bigoplus_{[v] \in G\backslash E^0}\sum_{[e] \in G_v\backslash vE^1}(\theta_e)_{*,i} = \alpha_i = \sum_{[e] \in G\backslash E^1}(\theta_e)_{*,i},\]
	which follows from the fact that $G_v\backslash vE^1 \cong \{[e] \in G\backslash E^1 : r(e) \in G\cdot v\}$.
\end{proof}

\begin{proof}[Proof of Theorem~\ref{thm:multitree sequence}]
	Since $\partial E$ is a closed subspace of $\Omega$ and $E^0$ is its complement, we have a short exact sequence
	\[0 \to C_0(E^0) \to C_0(\Omega) \to C_0(\partial E) \to 0,\]
	which induces a short exact sequence of full crossed products
	\[0 \to C_0(E^0) \rtimes_{\tau_0} G \xrightarrow{\iota} C_0(\Omega) \rtimes_{\tau} G \to C_0(\partial E) \rtimes_{\tau_\partial} G \to 0.\]
	Now note the following. By Remark~\ref{rem:amenabilityassumptions}, amenability of the action $G \curvearrowright \Omega$ implies that the actions $G \curvearrowright E^0$ and $G \curvearrowright \partial E$ are amenable. Then \cite[Theorem~5.15]{BEW} gives that all three actions $\tau_0$, $\tau$ and $\tau_\partial$ of $G$ on $C_0(E^0)$, $C_0(\Omega)$ and $C_0(\partial E)$ respectively are amenable. Finally, since $G$ is discrete, amenability of these actions implies that the corresponding full crossed products are isomorphic to the respective reduced crossed products \cite[Proposition 4.8]{AD}. Altogether, this implies that the sequence of reduced crossed products
	\[0 \to C_0(E^0) \rtimes_{\tau_0,r} G \xrightarrow{\iota} C_0(\Omega) \rtimes_{\tau,r} G \to C_0(\partial E) \rtimes_{\tau_\partial,r} G \to 0\]
	is also exact, where $\iota\colon C_0(E^0) \rtimes_r G \to C_0(\Omega) \rtimes_r G$ is the canonical inclusion. This gives the following six-term exact sequence in $K$-theory:
	\[\begin{tikzcd}
		K_0(C_0(E^0) \rtimes_{\tau_0,r} G)
		\arrow[r,"\iota_{*,0}"] &
		K_0(C_0(\Omega) \rtimes_{\tau,r} G) \arrow[r] &
		K_0(C_0(\partial E) \rtimes_{\tau_\partial,r} G) \arrow[d] \\
		K_1(C_0(\partial E) \rtimes_{\tau_\partial,r} G) \arrow[u] &
		K_1(C_0(\Omega) \rtimes_{\tau,r} G) \arrow[l] &
		K_1(C_0(E^0) \rtimes_{\tau_0,r} G)
		\arrow[l,"\iota_{*,1}"]
	\end{tikzcd}\]
	Then Propositions~\ref{prop:multitree Omega K-theory} and~\ref{prop:multitree iota_* diagram} together imply the result.
\end{proof}

\begin{remark}
	We note that the amenability assumptions in Theorem~\ref{thm:multitree sequence} are probably stronger than necessary for the statement. All that is strictly necessary for the proof is that $G$ satisfies the Baum--Connes conjecture with coefficients in $C_0(E^0)$ and $C_0(\Omega)$ (so that Theorem~\ref{thm:mainCELthm} can be applied), and that the sequence of crossed products
	\[0 \to C_0(E^0) \rtimes_{\tau_0,r} G \xrightarrow{\iota} C_0(\Omega) \rtimes_{\tau,r} G \to C_0(\partial E) \rtimes_{\tau_\partial,r} G \to 0\]
	is exact.
\end{remark}

\subsection{Example: trivial vertex stabilisers}\label{subsec:trivialvertexstabs}

In this subsection and the next, we provide more specialised $K$-theory formulas for the crossed product $C_0(\partial E) \rtimes_{\tau_\partial,r} G$ in the case where all vertex stabilisers are trivial (equivalently where $G$ acts freely on $E$), and in the case where all vertex stabilisers are infinite cyclic.

\begin{thm}\label{thm:matrixformulas}
	Let $G$ be a countable discrete group acting \textit{freely} on a row-finite, finitely aligned multitree $E$ with no sources. Suppose that the action $G \curvearrowright \Omega$ is amenable. Write $\Gamma := G\backslash E$ for the quotient directed graph, and let $A \in M_{\Gamma^0}(\Z)$ be the adjacency matrix of $\Gamma$ (defined by $A_{v,w} = \# v\Gamma^1 w$). The $K$-theory of the reduced crossed product $C_0(\partial E) \rtimes_{\tau_\partial,r} G$ is given by
	\[K_0(C_0(\partial E) \rtimes_{\tau_\partial,r} G) \cong \coker(1-A^T) \quad \text{and} \quad K_1(C_0(\partial E) \rtimes_{\tau_\partial,r} G) \cong \ker(1-A^T),\]
    where we consider $1-A^T$ to be a homomorphism $\Z[\Gamma^0] \to \Z[\Gamma^0]$.
\end{thm}

\begin{proof}
	Since $G$ acts freely on $E$, all stabiliser subgroups are trivial. Note that the (reduced) $C^*$-algebra of the trivial group is simply $\C$, and so for all $v \in E^0$, we have $C^*_r(G_v) \cong \C$. So (writing $1_v$ for the unique element in $G_v$) $K_0(C^*_r(G_v)) = \Z[1_v]$ and $K_1(C^*_r(G_v)) = 0$ for all $v \in E^0$. So applying Theorem~\ref{thm:multitree sequence}, we get the six-term exact sequence
	\[\begin{tikzcd}
		\bigoplus_{[v]\in G\backslash E^0} \Z[1_v] \arrow[r,"\id-\alpha_0"] &
		\bigoplus_{[v]\in G\backslash E^0} \Z[1_v] \arrow[r] &
		K_0(C_0(\partial E) \rtimes_{\tau_\partial,r} G) \arrow[d] \\
		K_1(C_0(\partial E) \rtimes_{\tau_\partial,r} G) \arrow[u] &
		0 \arrow[l] &
		0, \arrow[l]
	\end{tikzcd}\]
	where
	\[\alpha_0 = \sum_{[e]\in G\backslash E^1} (\theta_e)_{*,0}.\]
	For each $e \in E^1$, the map $\theta_e$ sends $1_{r(e)}$ to $1_{s(e)}$, and so the induced map on $K_0$ satisfies $(\theta_e)_{*,0}([1_{r(e)}]) = [1_{s(e)}]$. From this, it is easy to see that $\alpha_0$ is exactly (left multiplication by the matrix) $A^T$.
\end{proof}

\begin{remark}
    We can recover the well-known formula for the $K$-theory of directed graph $C^*$-algebras \cite{RaeSzy} (at least in the row-finite, no sources case) from the above theorem, using that a directed graph algebra is Morita equivalent to the crossed product of a free group acting on the boundary of the covering tree of the graph \cite[Corollary 4.14]{Kumjian-Pask}. The underlying action of the group on the covering tree is free \cite[Lemma~4.10]{Kumjian-Pask}, which implies that all vertex stabilisers are trivial; Proposition~\ref{prop:directed tree amenability} then gives that the action on the boundary is also amenable. Finally, Remark~\ref{rem:amenabilityassumptions} implies that the action on $\Omega$ is amenable, thereby allowing us to apply Theorem~\ref{thm:matrixformulas}.
\end{remark}

\subsection{Example: infinite cyclic vertex stabilisers}\label{subsec:icvertexstabs}

We now move to the case where all vertex stabilisers (and hence all edge stabilisers) are isomorphic to $\Z$. First we introduce some notation.

\begin{notation}\label{not:omegas}
    Let $G$ be a discrete group acting on a row-finite, finitely aligned multitree $E$ with no sources. For any path (finite or infinite) $\alpha$ in $E$, write $G_\alpha$ for the subgroup of $G$ which fixes $\alpha$. We write $\Gamma := G \backslash E$ for the quotient directed graph. For each vertex $v \in \Gamma^0$ we choose a lift $v' \in E^0$ of $v$, and for each edge $e \in \Gamma^1$ we choose a lift $e' \in E^1$ of $e$. For each $v \in \Gamma^0$ and $e \in \Gamma^1$ we define groups $G_v$ and $G_e$ by $G_v := G_{v'}$ and $G_e := G_{e'}$.
    
    For each $e \in \Gamma^1$, the group $G_e$ embeds into the group $G_{r(e)}$ in the following way. Write $v = r(e)$, and let $v' \in E^0$ be the chosen lift of $v$. The vertex $r(e') \in E^0$ is also a lift of $v$, and so will be in the same vertex orbit as $v$. In particular, there is some group element $g \in G$ such that $g \cdot r(e') = v'$. Then $G_{v'} = g G_{r(e')} g^{-1}$, and so we can define an injective homomorphism $G_e \into G_{r(e)}$ by $h \mapsto ghg^{-1}$. Replacing `$r$'s with `$s$'s in the above argument shows that $G_e$ also embeds into $G_{s(e)}$ in a similar way. We will often identify $G_e$ with its image in $G_{r(e)}$ or $G_{s(e)}$ under the respective embeddings.
    
    We now introduce additional notation for the special situation where for each $v' \in E^0$, the group $G_{v'}$ is infinite cyclic. For $e \in \Gamma^1$, we have by [Section~\ref{subsec:groupactionsongraphs}(3)] that $G_e$ is a subgroup of $G_{r(e)}$ of finite index, and hence is also infinite cyclic. For each $v \in \Gamma^0$ choose a generator $1_{v'}$ of $G_{v'} \cong \Z$ (where $v'$ is the chosen lift of $v$). Similarly, for each edge $e \in \Gamma^1$ we choose a generator $1_{e'}$ of $G_{e'} \cong \Z$ (where $e'$ is the chosen lift of $e$).
    
    Finally, to each $e \in \Gamma^1$ we associate integers $\omega_e, \omega_{\overline{e}}$ in the following way: we denote by $\omega_e$ the element of $\Z$ such that the embedding of $G_e$ into $G_{r(e)}$ sends $1_{e'}$ to $\omega_e 1_{r(e)'}$, and similarly we denote by $\omega_{\overline{e}}$ the element of $\Z$ such that the embedding of $G_e$ into $G_{s(e)}$ sends $1_{e'} $ to $\omega_{\overline{e}} 1_{s(e)'}$.
\end{notation}

\begin{thm} \label{thm:directed GBS Ktheory}
	Let $G$ be a countable discrete group acting on a row-finite, finitely aligned multitree $E$ with no sources. Suppose that for each $v \in E^0$, the stabiliser $G_v \leq G$ is isomorphic to $\Z$. Suppose further that the action $G \curvearrowright \Omega$ is amenable. Define maps $A_0, A_1 \colon \Z[\Gamma^0] \to \Z[\Gamma^0]$ by
    \[A_0(v) = \sum_{e \in v\Gamma^1} \abs{\omega_e}s(e) \qquad \text{and} \qquad A_1(v) = \sum_{e \in v\Gamma^1} \sgn(\omega_e)\omega_{\overline{e}}s(e)\]
    for $v \in \Gamma^0$. We have
	\[K_0(C_0(\partial E) \rtimes_{\tau_\partial,r} G) \cong \coker(1-A_0) \oplus \ker(1-A_1)
	\]
	 and
    \[ K_1(C_0(\partial E) \rtimes_{\tau_\partial,r} G) \cong \coker(1-A_1) \oplus \ker(1-A_0).\]
\end{thm}

\begin{proof}
	Recall that $K_0(C^*_r(\Z)) = \Z[u_0]_0$ and $K_1(C^*_r(\Z)) = \Z[u_1]_1$. For $v \in E^0$ and $n \in \Z$, write $u^{(v)}_n \in C^*_r(G_v)$ for the unitary representation of $n \in \Z \cong G_v$; we omit the superscript `$(v)$' when it is clear what the vertex is. So for all $v \in E^0$, we have
	\[K_0(C^*_r(G_v)) = \Z[u^{(v)}_0]_0, \qquad K_1(C^*_r(G_v)) = \Z[u^{(v)}_1]_1.\]
	From Theorem~\ref{thm:multitree sequence}, we get the six-term exact sequence
	\[\begin{tikzcd}
		\bigoplus_{[v]\in G\backslash E^0} \Z[u^{(v)}_0]_0 \arrow[r,"\id-\alpha_0"] &
		\bigoplus_{[v]\in G\backslash E^0} \Z[u^{(v)}_0]_0 \arrow[r] &
		K_0(C_0(\partial E) \rtimes_{\tau_\partial,r} G) \arrow[d] \\
		K_1(C_0(\partial E) \rtimes_{\tau_\partial,r} G) \arrow[u] &
		\bigoplus_{[v]\in G\backslash E^0} \Z[u^{(v)}_1]_1 \arrow[l] &
		\bigoplus_{[v]\in G\backslash E^0} \Z[u^{(v)}_1]_1, \arrow[l,"\id-\alpha_1"]
	\end{tikzcd}\]
	where
	\[\alpha_i = \sum_{[e]\in G\backslash E^1} (\theta_e)_{*,i}.\]
	Now $\theta_e$ behaves in the following way. Fix $e \in E^1$. We know that $G_e$ embeds into $G_{r(e)}$ as the subgroup $\omega_e G_{r(e)}$, so we can choose $\Sigma_e$ to be the set $\{0, 1, \dots, \abs{\omega_e} - 1\}$ (where we have identified $G_{r(e)}$ with $\Z$). Then the map
    \[ \theta_e \colon C^*_r(G_{r(e)}) \to \mathcal{K}(\ell^2(\Sigma_e)) \otimes C^*_r(G_{s(e)}) \cong M_{\abs{\omega_e}}(C^*_r(G_{s(e)})) \]
    satisfies
    \[ \theta_e(u^{(r(e))}_0) = \sum_{h=0}^{\abs{\omega_e}-1} e_{h,h} \otimes u^{(s(e))}_0 = \diag(u^{(s(e))}_0, \dots, u^{(s(e))}_0) \]
    and
    \[ \theta_e(u^{(r(e))}_1) = \sum_{h=0}^{\abs{\omega_e}-1} e_{h,(-1)\cdot h} \otimes u^{(s(e))}_{-h+1+((-1)\cdot h)} = e_{0,\abs{\omega_e}-1} \otimes u^{(s(e))}_{\abs{\omega_e}1_{r(e)}} + \sum_{h=1}^{\abs{\omega_e}-1} e_{h,h-1} \otimes u^{(s(e))}_0. \]
    Now note that $\abs{\omega_e}1_{r(e)} = \frac{\abs{\omega_e}}{\omega_e}1_e = \sgn(\omega_e)1_e = \sgn(\omega_e)\omega_{\overline{e}}1_{s(e)}$, so $u^{(s(e))}_{\abs{\omega_e}1_{r(e)}} = u^{(s(e))}_{\sgn(\omega_e)\omega_{\overline{e}}}$. This means that
    \[ \theta_e(u^{(r(e))}_1) = e_{0,\abs{\omega_e}-1} \otimes u^{(s(e))}_{\sgn(\omega_e)\omega_{\overline{e}}} + \sum_{h=1}^{\abs{\omega_e}-1} e_{h,h-1} \otimes u^{(s(e))}_0 = 
        \begin{pmatrix}0 & u^{(s(e))}_{\sgn(\omega_e)\omega_{\overline{e}}} \\ \diag(u^{(s(e))}_0, \dots, u^{(s(e))}_0) & 0\end{pmatrix} \]
    in block matrix form. But this matrix is homotopy equivalent to the diagonal matrix $\diag(u^{(s(e))}_{\sgn(\omega_e)\omega_{\overline{e}}},u^{(s(e))}_0,\dots, u^{(s(e))}_0)$.
    
    To summarise, we have
    \[ \theta_e(u_0) = \diag(u_0, \dots, u_0) \qquad \text{and} \qquad \theta_e(u_1) \sim_h \diag(u_{\sgn(\omega_e)\omega_{\overline{e}}},u_0,\dots, u_0), \]
    so the induced maps on $K$-theory are given by
	\begin{align*}
		(\theta_e)_{*,0}\colon \Z[u_0]_0 &\to \Z[u_0]_0 & (\theta_e)_{*,1}\colon \Z[u_1]_1 &\to \Z[u_1]_1 \\
		[u_0]_0 &\mapsto \abs{\omega_e}[u_0]_0, & [u_1]_1 &\mapsto [u_{\sgn(\omega_e)\omega_{\overline{e}}}]_1 = \sgn(\omega_e)\omega_{\overline{e}}[u_1]_1,
	\end{align*}
	and hence $(\theta_e)_{*,0}$ is multiplication by $\abs{\omega_e}$, and $(\theta_e)_{*,1}$ is multiplication by $\sgn(\omega_e)\omega_{\overline{e}}$. The theorem then follows from the commutativity of the diagram
	\[\begin{tikzcd}
		\bigoplus_{[v] \in G\backslash E^0} K_i(C^*(G_v)) \arrow[r,"\alpha_i"] \arrow[d,"\cong"] & \bigoplus_{[v] \in G\backslash E^0} K_i(C^*(G_v)) \arrow[d,"\cong"] \\
		\Z[\Gamma^0] \arrow[r,"A_i"] & \Z[\Gamma^0]
	   \end{tikzcd}\]
	for $i=0,1$.
\end{proof}

\section{Group actions on undirected trees}\label{sec:undirectedtrees}

We can also use Theorem~\ref{thm:multitree sequence} to study crossed products coming from actions on the boundaries of \textit{undirected} trees. Such actions arise naturally for example in the study of graph-of-groups $C^*$-algebras (see \cite{BMPST}, in particular Theorem 4.1).

In Section~\ref{subsec:dual graphs} we discuss the notion of a dual graph, and how it brings group actions on undirected trees into our setting of group actions on multitrees. Then, in Section~\ref{subsec:undirected tree K-theory}, we explain how Theorem~\ref{thm:multitree sequence} can be applied to yield a six-term exact sequence for the $K$-theory of crossed products induced by a group action on an undirected tree.

\subsection{Dual graphs} \label{subsec:dual graphs}

In this subsection, we explain how we can use a dual graph construction to associate to any undirected tree a suitable multitree. We mirror the construction of the dual graph for directed graphs (see, for example, \cite[Page~313]{BPRS}), except that we only consider \textit{reduced} paths.

\begin{defin}
	Let $\Gamma$ be a undirected graph. The \textit{dual graph} of $\Gamma$ is the directed graph $\widehat{\Gamma} = (\widehat{\Gamma}^0, \widehat{\Gamma}^1, \widehat{r}, \widehat{s})$ defined by
	\[\widehat{\Gamma}^0 = \Gamma^1, \quad \widehat{\Gamma}^1 = \Gamma^2_r, \quad \widehat{r}(ef) = e, \quad \widehat{s}(ef) = f, \quad ef \in \Gamma^2_r.\]
\end{defin}

\begin{eg}
    Below is an example of an undirected tree (on the left), along with its dual multitree (on the right).
    
    
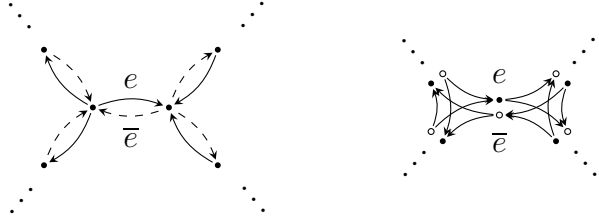
\[
        \begin{tikzpicture}[baseline=(current bounding box.center)]
            \node[vertex] (a) at (0,0) {};
            \node[vertex,label={[label distance=0.5em,rotate=130]right:{$\cdots$}}] (a1) at (130:1) {};
            \node[vertex,label={[label distance=0.5em,rotate=-130]right:{$\cdots$}}] (a2) at (-130:1) {};
            \node[vertex] (b) at (1,0) {};
            \node[vertex,label={[label distance=0.5em,rotate=50]right:{$\cdots$}}] (b1) at ($(b) + (50:1)$) {};
            \node[vertex,label={[label distance=0.5em,rotate=-50]right:{$\cdots$}}] (b2) at ($(b) + (-50:1)$) {};
            
            \draw[forwards]
            (a) edge (a1)
            (a) edge (a2)
            (a) edge node[midway,above]{$e$} (b)
            (b1) edge (b)
            (b2) edge (b);
            
            \draw[backwards]
            (a) edge (a1)
            (a) edge (a2)
            (a) edge node[midway,below]{$\overline{e}$} (b)
            (b1) edge (b)
            (b2) edge (b);
        \end{tikzpicture}
        \hspace{3em}
        \begin{tikzpicture}[baseline=(current bounding box.center)]
            \node[vertex,label=above:{$e$}] (e) at (0,0.1) {};
            \node[open-vertex,label=below:{$\overline{e}$}] (e-) at (0,-0.1) {};
            
            \node[vertex] (f1) at ($(0.5,0) + (50:0.5) + (-40:0.1)$) {}; 
            \node[open-vertex] (f1-) at ($(0.5,0) + (50:0.5) - (-40:0.1)$) {};
            \node at ($(0.5,0) + (50:1)$) {\rotatebox{50}{$\cdots$}};
            
            \node[vertex] (f2) at ($(-0.5,0) + (130:0.5) + (220:0.1)$) {}; 
            \node[open-vertex] (f2-) at ($(-0.5,0) + (130:0.5) - (220:0.1)$) {};
            \node at ($(-0.5,0) + (130:1)$) {\rotatebox{130}{$\cdots$}};
            
            \node[vertex] (f3) at ($(-0.5,0) + (-130:0.5) - (-220:0.1)$) {}; 
            \node[open-vertex] (f3-) at ($(-0.5,0) + (-130:0.5) + (-220:0.1)$) {};
            \node at ($(-0.5,0) + (-130:1)$) {\rotatebox{-130}{$\cdots$}};
            
            \node[vertex] (f4) at ($(0.5,0) + (-50:0.5) - (40:0.1)$) {}; 
            \node[open-vertex] (f4-) at ($(0.5,0) + (-50:0.5) + (40:0.1)$) {};
            \node at ($(0.5,0) + (-50:1)$) {\rotatebox{-50}{$\cdots$}};
            
            \draw[-stealth, bend left=20]
            (e) edge (f4-)
            (f1) edge (e-)
            (f4) edge (f1-)
            
            (e-) edge (f2)
            (f2-) edge (f3)
            (f3-) edge (e);
            
            \draw[-stealth, bend right=20]
            (e) edge (f1-)
            (f1) edge (f4-)
            (f4) edge (e-)
            
            (e-) edge (f3)
            (f2-) edge (e)
            (f3-) edge (f2);
        \end{tikzpicture}
    \]
\end{eg}

Note that there is a natural identification of paths (finite or infinite) in $\widehat{\Gamma}$ with reduced paths (finite or infinite) of length at least one in $\Gamma$, given by viewing vertices in $\widehat{\Gamma}$ as edges in $\Gamma$. We will use this identification in the proof of the following proposition, which explains why this dual graph construction is useful.

\begin{prop} \label{prop:tree dual multitree}
	Let $T$ be a locally finite nonsingular undirected tree. Its dual graph $\widehat{T}$ is a row-finite, finitely aligned multitree with no sources, and $\partial T \cong \partial \widehat{T}$.
	
	Suppose further that a discrete group $G$ acts on $T$. Then $G$ acts naturally on $\partial \widehat{T}$ and $G \curvearrowright \partial T \cong G \curvearrowright \partial \widehat{T}$.
\end{prop}

\begin{proof}
	First we show that $\widehat{T}$ is a multitree. Take two vertices $e, f \in \widehat{T}^0$, and suppose for a contradiction that there are two distinct paths $\lambda, \mu \in e \widehat{T}^* f$. This implies that there are two distinct reduced paths in $T$ from $r(f)$ to $s(e)$, contradicting that $T$ is a tree. Hence there is at most one directed path between any two vertices in $\widehat{T}^0$, so $\widehat{T}$ is a multitree.
	
	To see that $\widehat{T}$ is row-finite and has no sources, note that for any vertex $e \in \widehat{T}^0$, the set $\widehat{r}^{-1}(e) \subseteq \widehat{T}^1$ is in a one-to-one correspondence with the set $s(e)T^1\setminus\{\overline{e}\}$. In particular, we have that $\abs{\widehat{r}^{-1}(e)} = \abs{s(e)T^1} - 1$. So the assumption that $T$ is locally finite and nonsingular implies that $\widehat{T}$ is row-finite and has no sources.
	
	To see that $\widehat{T}$ is finitely aligned, fix $e, f \in \widehat{T}^0$. We need to show that the set of minimal common upper bounds for $e$ and $f$ is finite. If $e \leq f$ (respectively $f \leq e$) then the only minimal common upper bound for $e$ and $f$ is $f$ (respectively $e$), so we are done. So suppose that $e$ and $f$ are not comparable but have common upper bounds. Take some common upper bound $x$ of $e$ and $f$, so there are paths $\lambda \in e\widehat{T}^* x$ and $\mu \in f\widehat{T}^* x$, which we will view as reduced paths in $T$. Write $\nu$ for the longest path such that $\lambda = \alpha\nu$ and $\mu = \beta\nu$ for some paths $\alpha$ and $\beta$. Then, writing $\overline{\beta}$ for the reversal of the path $\beta$, we have that $\alpha\overline{\beta}$ is the (unique) reduced path in $T$ from $r(f)$ to $r(e)$, and note that $r(\nu) = s(\alpha)$ lies on this path. Write $y$ for the range-most edge in $\nu$. Then $y$ is a common upper bound of $e$ and $f$, and we have $y \leq x$. This shows that any common upper bound of $e$ and $f$ lies above some common upper bound $y$ with $r(y) \in T^0$ lying on the reduced path in $T$ from $r(f)$ to $r(e)$. But this path is finite, and since $T$ is locally finite, there are only finitely many such $y$'s. Hence $e$ and $f$ have only finitely many minimal common upper bounds.
	
	We now show that $\partial T \cong \partial \widehat{T}$. First notice that the sets $T^\infty_r$ and $\widehat{T}^\infty$ are isomorphic under the natural identification. This isomorphism respects shift-tail equivalence, and so the boundaries $\partial T$ and $\partial \widehat{T}$ are isomorphic as sets. Moreover, for any edge $e \in T^1$, the cylinder set $Z_\partial(e) \subseteq \partial T$ is sent by the isomorphism to the cylinder set $Z_\partial(e) \subseteq \partial \widehat{T}$, where $e$ is considered as a vertex of $\widehat{T}$. Thus $\partial T \cong \partial \widehat{T}$ as topological spaces.
	
	Finally, suppose that a discrete group $G$ acts on $T$. This induces an action of $G$ on $\widehat{T}$ in the natural way, which in turn induces an action of $G$ on $\partial \widehat{T}$, and the isomorphism $\partial T \cong \partial \widehat{T}$ described above intertwines the actions of $G$ on $\partial T$ and $\partial \widehat{T}$, which implies that $G \curvearrowright \partial T \cong G \curvearrowright \partial \widehat{T}$, as required. 
\end{proof}

\subsection{Group actions on boundaries of undirected trees} \label{subsec:undirected tree K-theory}

In this subsection we apply Theorem~\ref{thm:multitree sequence} to arrive at the following result for groups acting on the boundaries of \textit{undirected} trees. Note that we write $G_{ef}$ for the stabiliser subgroup of $G$ at $ef \in T^2_r = \widehat{T}^1$, and the map $\theta_{ef}$ is defined as in Notation~\ref{not:thetaetc}, where $ef$ is considered as an edge in the dual multitree $\widehat{T}$.

\begin{thm} \label{thm:bidirected tree sequence}
	Let $G$ be a countable discrete group acting on a locally finite nonsingular undirected tree $T$. Suppose that $G_v$ is amenable for all $v \in T^0$. Then for $\widehat{\alpha}_i = \sum_{[ef] \in G\backslash T^2_r} (\theta_{ef})_{*,i}$ we have the following six-term exact sequence
	\[\begin{tikzcd}
		\bigoplus_{[e]\in G\backslash T^1} K_0(C^*_r(G_e))
		\arrow[r,"\id-\widehat{\alpha}_0"] &
		\bigoplus_{[ef]\in G\backslash T^2_r} K_0(C^*_r(G_{ef})) \arrow[r] &
		K_0(C_0(\partial T) \rtimes_{\tau_\partial,r} G) \arrow[d] \\
		K_1(C_0(\partial T) \rtimes_{\tau_\partial,r} G) \arrow[u] &
		\bigoplus_{[ef]\in G\backslash T^2_r} K_1(C^*_r(G_{ef})) \arrow[l] &
		\bigoplus_{[e]\in G\backslash T^1} K_1(C^*_r(G_e)).
		\arrow[l,"\id-\widehat{\alpha}_1"]
	\end{tikzcd}\]
\end{thm}

\begin{proof}
	By Proposition~\ref{prop:tree dual multitree}, we have that the dual graph of $T$ is a row-finite, finitely aligned multitree with no sources. The vertex stabilisers of $\widehat{T}$ are the edge stabilisers of $T$, which are amenable because they are subgroups of the vertex stabilisers of $T$, which are assumed to be amenable. Moreover, the action of $G$ on $\partial \widehat{T} \cong \partial T$ is amenable by \cite[Proposition~5.2.1 and Lemma~5.2.6]{Brown-Ozawa}. So Remark~\ref{rem:amenabilityassumptions} allows us to apply Theorem~\ref{thm:multitree sequence} for the dual graph to get a six-term exact sequence
	\[\begin{tikzcd}
		\bigoplus_{[e]\in G\backslash \widehat{T}^0} K_0(C^*_r(G_e))
		\arrow[r,"\id-\widehat{\alpha}_0"] &
		\bigoplus_{[ef]\in G\backslash \widehat{T}^1} K_0(C^*_r(G_{ef})) \arrow[r] &
		K_0(C_0(\partial \widehat{T}) \rtimes_{\tau_\partial,r} G) \arrow[d] \\
		K_1(C_0(\partial \widehat{T}) \rtimes_{\tau_\partial,r} G) \arrow[u] &
		\bigoplus_{[ef]\in G\backslash \widehat{T}^1} K_1(C^*_r(G_{ef})) \arrow[l] &
		\bigoplus_{[e]\in G\backslash \widehat{T}^0} K_1(C^*_r(G_e)),
		\arrow[l,"\id-\widehat{\alpha}_1"]
	\end{tikzcd}\]
	where $\widehat{\alpha}_i = \sum_{[ef] \in G\backslash\widehat{T}^1}(\theta_{ef})_{*,i}$. But $\widehat{T}^0 = T^1$ and $\widehat{T}^1 = T^2_r$ by definition, and Proposition~\ref{prop:tree dual multitree} implies that $C_0(\partial \widehat{T}) \rtimes_{\tau_\partial,r} G \cong C_0(\partial T) \rtimes_{\tau_\partial,r} G$, which together imply the theorem.
\end{proof}

\begin{remark}
	The assumption in Theorem~\ref{thm:bidirected tree sequence} that all vertex stabilisers are amenable means that the reduced group $C^*$-algebras and crossed products are naturally isomorphic to their respective \textit{full} versions. Thus, Theorem~\ref{thm:bidirected tree sequence} can be viewed as a special case of \cite[Corollary~6.4]{Mundey-Rennie}, which provides the analogous six-term exact sequence for the full group $C^*$-algebras and crossed products, without needing any amenability assumptions. However, the maps in the sequence in \cite{Mundey-Rennie} are described in terms of the $KK$-classes of $C^*$-correspondences rather than in terms of induced homomorphisms. It is not immediately clear to the authors how these pictures correspond to each other. However, we note that in the special cases where all vertex stabilisers are trivial or where all vertex stabilisers are infinite cyclic, the formulas for $\id-\widehat{\alpha}_i $ in Theorem~\ref{thm:bidirected tree sequence} reduce to the formulas for $\Lambda_i$ in \cite[Section~6]{Mundey-Rennie}.
\end{remark}

\begin{remark}
In \cite{BMPST}, the authors associate to each locally finite nonsingular graph of discrete groups $\GG$ a universal $C^*$-algebra $C^*(\GG)$, and show that $C^*(\GG)$ is stably isomorphic to the full crossed product of the fundamental group of $\GG$ acting on the Bass--Serre tree of $\GG$ \cite[Theorem~4.1]{BMPST}. This means that in the setting where all vertex stabilisers are amenable, Theorem~\ref{thm:bidirected tree sequence} can be rephrased in terms of the $K$-theory of the graph-of-groups $C^*$-algebra $C^*(\GG)$, just as is done in \cite[Theorem~6.2]{Mundey-Rennie}.
\end{remark}

\section{Properties of the action $G \curvearrowright \partial E$}\label{sec:propsofaction}

In this section we examine dynamical properties for the action of a discrete group $G$ on the boundary $\partial E$ of a row-finite, finitely aligned multitree $E$ with no sources, induced from an action of $G$ on $E$. We investigate minimality in Section~\ref{subsec:minimality}, local contractivity in Section~\ref{subsec:local contractivity}, topological freeness in Section~\ref{subsec:topological freeness}, and finally amenability in Section~\ref{subsec:amenability}. Note that it follows from Proposition~\ref{prop:tree dual multitree} that the results in this section generalise the results in \cite[Section~5]{BMPST}, which deals only with group actions on undirected trees. We will use the same notation as in Notation~\ref{not:omegas}.

\subsection{Minimality} \label{subsec:minimality}

Recall that an action of a discrete group $G$ on a locally compact Hausdorff space $X$ is called \textit{minimal} if every orbit of points of $X$ is dense in $X$. We have the following characterisation in our context.

\begin{prop}
	Suppose that a discrete group $G$ acts on a row-finite, finitely aligned multitree $E$ with no sources. The induced action $G \curvearrowright \partial E$ is minimal if and only if the graph $\Gamma$ is cofinal, that is, for every infinite path $\lambda \in \Gamma^\infty$ and every vertex $v \in \Gamma^0$, there is some vertex $w$ on $\lambda$ so that there is a path in $\Gamma$ from $w$ to $v$.
\end{prop}

\begin{proof}
	First suppose that the action $G \curvearrowright \partial E$ is minimal. Fix an infinite path $\lambda$ in $\Gamma$ and a vertex $v \in \Gamma^0$. Then $\lambda$ and $v$ lift to some infinite path $\lambda' \in E^\infty$ and some vertex $v' \in E^0$ respectively. By minimality, there is $g \in G$ so that $g \cdot [\lambda'] = [g \cdot \lambda'] \in Z_\partial(v')$. This means there is an infinite path $\mu' \in v'E^\infty$ so that $\mu'$ and $g \cdot \lambda'$ agree after some vertex, say $w' \in E^0$. So there is a path from $w'$ (which is on $g \cdot \lambda'$) to $v'$, which descends to a path from $w \in \Gamma^0$ to $v$. But since $w'$ is in the same orbit as $g^{-1} \cdot w'$, which is on $\lambda'$, we get that $w$ is on $\lambda$, so there is a path from some $w$ on $\lambda$ to $v$.
	
	Now suppose that $\Gamma$ is cofinal. Fix $\lambda' \in E^\infty$ and $v' \in E^0$ arbitrary. It is enough to prove that there is $g \in G$ such that $g \cdot [\lambda'] = [g \cdot \lambda'] \in Z_\partial(v')$. Write $\lambda$ and $v$ for the images of $\lambda'$ and $v'$ respectively under the quotient map. Then $\lambda$ is an infinite path in $\Gamma$ and $v \in \Gamma^0$ is a vertex, so there is some vertex $w$ on $\lambda$ such that there is a path in $\Gamma$ from $w$ to $v$. Now, there is a lift $w' \in E^0$ of $w$ such that $w'$ is on $\lambda'$. Also, the path from $w$ to $v$ lifts to a path from some vertex $w'' \in E^0$ to $v'$, and $w''$ is in the same vertex orbit as $w'$. Let $g \in G$ be an element such that $g \cdot w' = w''$. Then $w''$ will be on $g \cdot \lambda'$, so $g \cdot \lambda'$ is shift-tail equivalent to an infinite path with range $v'$. Hence $g \cdot [\lambda'] \in Z_\partial(v')$.
\end{proof}

\subsection{Local contractivity} \label{subsec:local contractivity}

Recall that an action of a discrete group $G$ on a locally compact Hausdorff space $X$ is called \textit{locally contractive} if for every non-empty open set $U \subseteq X$ there is a non-empty open set $V \subseteq U$ and an element $g \in G$ such that $g \overline{V} \subsetneq V$. Like in \cite{BMPST}, we give a sufficient condition for the action $G \curvearrowright \partial E$ to be locally contractive.

First recall that a loop $\eta=\eta_1\cdots\eta_m$ in a directed graph $\Gamma$ has an \textit{entrance at vertex} $x\in\Gamma^0$ if there exists $\eta_i$ such that $r(\eta_i)=x$ and $|r^{-1}(x)|\ge 2$.

\begin{prop}\label{prop:localcon}
	Suppose that a discrete group $G$ acts on a row-finite, finitely aligned multitree $E$ with no sources. Suppose that for every vertex $v \in \Gamma^0$ there is a loop $\eta=\eta_1\dots\eta_m$ in $\Gamma$ such that 
	\begin{itemize}
		\item[(i)] $\eta$ has an entrance, or $[G_{r(\eta_i)}:G_{\eta_i}]\ge 2$ for some $1\le i\le m$; and
		\item[(ii)]  there is a path from $r(\eta)$ to $v$.
	\end{itemize}
Then the induced action  $G \curvearrowright \partial E$ is locally contractive.
\end{prop}

\begin{proof}
	Let $U \subseteq \partial E$ be a non-empty open subset, and let $v' \in E^0$ be such that $Z_\partial(v') \subseteq U$. Let $v \in \Gamma^0$ be the image of $v'$ under the quotient map. Let $\eta$ be as in the assumption. The path from $r(\eta)$ to $v$ lifts to a path from some vertex $w' \in E^0$ to $v'$, and so $Z_\partial(w') \subseteq Z_\partial(v') \subseteq U$ by Lemma~\ref{lemma:multitree cylinder sets}. Now $\eta$ lifts to some path $\eta'$ from some vertex $w'' \in E^0$ to $w'$, and $w''$ is in the vertex orbit of $w'$. So $w'' = g \cdot w'$ for some $g \in G$, and $g \cdot Z_\partial(w') = Z_\partial(g \cdot w') \subseteq Z_\partial(w')$ by Lemma~\ref{lemma:multitree cylinder sets}. Note that $Z_\partial(w')$ is closed, so all that is left to show is that the containment $Z(g \cdot w') \subseteq Z(w')$ is proper.
    
    By assumption, there is $x\in\Gamma^0$ such that $\eta$ either has an entrance at $x$, or that there is $\eta_i$ with $r(\eta_i)=x$ and $[G_x:G_{\eta_i}]\ge 2$. Write $x' \in E^0$ for a lift of $x$ which lies on $\eta'$. Note that by the orbit-stabiliser theorem, $\sum_{e \in x\Gamma^1}[G_x:G_e]$ is equal to the number of edges in $E^1$ with range $x'$; and the assumption ensures that this number is at least $2$. In particular, if $e'$ is the edge on $\eta'$ with range $x'$, then there is a distinct edge $e'' \in x'E^1$. Note that $w' \leq s(e'')$, so by Lemma~\ref{lemma:multitree cylinder sets} we have that $Z_\partial(s(e'')) \subseteq Z_\partial(w')$.
    
    We claim that $Z_\partial(s(e'')) \cap Z_\partial(g \cdot w') = \emptyset$. Indeed, suppose for a contradiction that $Z_\partial(s(e'')) \cap Z_\partial(g \cdot w') \neq \emptyset$. Then Proposition~\ref{prop:multitree cylinder intersection} implies that $s(e'')$ and $w'' = g \cdot w'$ have a common upper bound $u' \in E^0$. Write $\lambda_1, \lambda_2$ for the paths from $u'$ to $s(e'')$ and $w''$ respectively, and write $\mu$ for subpath of $\eta'$ from $w''$ to $x'$. Then $e''\lambda_1$ and $\mu\lambda_2$ are both paths from $u'$ to $x'$, and they are distinct since they have distinct range-most edges by construction. This contradicts that $E$ is a multitree, and so $Z_\partial(s(e'')) \cap Z_\partial(g \cdot w') = \emptyset$. This implies that the containment $Z_\partial(g \cdot w') \subseteq Z_\partial(w')$ is proper, completing the proof.
\end{proof}

\subsection{Topological freeness} \label{subsec:topological freeness}

Recall that an action of a group $G$ on a locally compact Hausdorff space $X$ is called \textit{topologically free} if for all $g \in G\setminus\{1\}$, the set $\{x \in X : g \cdot x \neq x\}$ is dense in $X$. For $G$ discrete, an action $G \curvearrowright X$ is topologically free if and only if the set of points in $X$ with trivial isotropy is dense in $X$.

It seems to be difficult in general to characterise topological freeness of an action $G \curvearrowright \partial E$ in terms of the corresponding quotient directed graph and the stabiliser subgroups. However, we have characterisations in the special cases where $G$ acts freely on $E$ (Proposition~\ref{prop:free action topological freeness}) and where all vertex stabilisers are infinite cyclic (Proposition~\ref{prop:GBS topological freeness}).

We start with a lemma.

\begin{lemma} \label{lemma:path lifts}
	Suppose that a discrete group $G$ acts on a row-finite, finitely aligned multitree $E$ with no sources. For any finite path $\alpha = e_1 e_2 \dots e_n$ in $\Gamma$ and any lift $v' \in E^0$ of $r(\alpha)$, the number of lifts of $\alpha$ with range $v'$ is equal to
	\[[G_{r(e_1)}:G_{e_1}] [G_{r(e_2)}:G_{e_2}]\cdots [G_{r(e_n)}:G_{e_n}].\]
\end{lemma}

\begin{proof}
	The claim follows from the fact that for any edge $e \in \Gamma^1$ and any lift $v' \in E^0$ of $r(e)$, the number of lifts of $e$ with range $v'$ is equal to $[G_{r(e)}:G_e]$.
\end{proof}

We start with the case where $G$ acts freely on $E$. First, recall that a directed graph $\Gamma$ is called \textit{aperiodic} if every loop has an entrance.

\begin{prop} \label{prop:free action topological freeness}
	Suppose that a discrete group $G$ acts \textit{freely} on a row-finite, finitely aligned multitree $E$ with no sources. The induced action $G \curvearrowright \partial E$ is topologically free if and only if $\Gamma$ is aperiodic.
\end{prop}

\begin{proof}
	First suppose that $\Gamma$ is aperiodic. This means that for every vertex $v \in \Gamma^0$, there is an infinite path $\lambda \in v\Gamma^\infty$ with aperiodic edge sequence \cite[Lemma~3.2]{BPRS} (note the opposite convention for path direction used in that paper). To show that $G \curvearrowright \partial E$ is topologically free, we prove that every cylinder set in $\partial E$ contains a point with trivial isotropy.
	
	Fix $v' \in E^0$. Write $v \in \Gamma^0$ for its image under the quotient map, and let $\lambda \in v\Gamma^\infty$ be an infinite path with aperiodic edge sequence. Let $\lambda' \in v'E^\infty$ be a lift of $\lambda$. We claim that $[\lambda']\in Z_\partial(v')$ has trivial isotropy. Indeed, suppose that $g \in G$ fixes $[\lambda']$, so there exist $m,n \in \N$ such that $\sigma^m(g\cdot\lambda') = \sigma^n(\lambda')$. This implies that $\sigma^m(\lambda)$ and $\sigma^n(\lambda)$ are the same infinite path in $\Gamma^\infty$, and so by aperiodicity of $\lambda$ we must have that $m=n$. This means that $g$ fixes the infinite path $\sigma^n(\lambda')$ and hence the vertex $r(\sigma^n(\lambda'))$. But since $G$ is assumed to act freely on $E$, this implies that $g$ is the identity element. Hence $[\lambda']$ has trivial isotropy as claimed.
	
	Now suppose that $\Gamma$ is not aperiodic. Then $G$ is not trivial (since if it were, then $\Gamma = G \backslash E$ would simply be the multitree $E$, which is aperiodic since it cannot contain any loops), and there is a loop $\eta$ in $\Gamma$ with no entrance. Let $v = r(\eta)$, and let $v' \in E^0$ be a lift of $v$. First we claim that there is a unique lift $\lambda' \in v'E^\infty$ of $\eta^\infty:=\eta\eta\cdots$ with range $v'$. Indeed, since all vertex stabilisers are trivial, we have $[G_{r(e)}:G_e] = 1$ for all $e \in \Gamma^1$, and so Lemma~\ref{lemma:path lifts} implies that for every lift $v'' \in E^0$ of $v$, there is a unique lift $\eta'' \in v''E^*$ of $\eta$. This in turn implies that $\lambda'$ is the only lift of $\eta^\infty$ with range $v'$, as claimed.
	
	Now since $\eta$ has no entrance, the only infinite path in $\Gamma$ with range $v$ is the infinite path $\eta^\infty$. This implies that $\lambda'$ is the only infinite path in $E$ with range $v'$, and so $Z_\partial(v')$ is the singleton set $\{[\lambda']\}$. We claim that $[\lambda']$ has non-trivial isotropy. Write $m$ for the length of $\eta$. Then $\sigma^m(\lambda')$ is the unique lift of $\eta^\infty$ with range $r(\sigma^m(\lambda'))$. Since $r(\sigma^m(\lambda'))$ is also a lift of $v \in \Gamma^0$, it is in the same orbit as $v'$, and so there is some (necessarily non-identity) $g \in G$ such that $g \cdot v' = r(\sigma^m(\lambda'))$. But $g \cdot \lambda'$ must be a lift of $\eta^\infty$ with range $r(\sigma^m(\lambda'))$, and so uniqueness gives that $g \cdot \lambda' = \sigma^m(\lambda')$. Hence $g \cdot [\lambda'] = [\lambda']$, showing that $[\lambda']$ has non-trivial isotropy.
	
	This implies that $Z_\partial(v')$ is an open set in $\partial E$ which does not contain a point with trivial isotropy. Hence the set of points in $\partial E$ with trivial isotropy is not dense in $\partial E$, so the action $G \curvearrowright \partial E$ is not topologically free.
\end{proof}

We now turn to the situation where the vertex stabilisers of the action $G \curvearrowright E$ are all infinite cyclic. Recall from Notation~\ref{not:omegas} the definition of the integers $\omega_e, \omega_{\overline{e}}$ for each $e \in \Gamma^1$. We need the following concepts from \cite{BMPST}, adapted to our context.

\begin{defin}
	For a path $\alpha = e_1 e_2 \dots e_n$ in $\Gamma$, we define the \textit{signed index ratio} $q(\alpha) \in \Q^\times$ by
	\[q(\alpha) = \prod_{i=1}^n \frac{\omega_{\overline{e_i}}}{\omega_{e_i}}.\]
\end{defin}

Consistently with \cite{BMPST}, for $q \in \Q$ we write $\langle q \rangle$ for the smallest positive integer denominator that can be used to express $q$ as a fraction. Recall that for any infinite path $\lambda = e_1 e_2 \cdots$ and any $n \geq 1$, we denote by $\lambda_n$ the path $e_1 \ldots e_n$. For any finite path $\alpha = e_1 e_2 \dots e_k$ and any $1 \leq n \leq k$, we similarly denote by $\alpha_n$ the path $e_1 \dots e_n$.

\begin{prop} \label{prop:GBS topological freeness}
	Suppose that a discrete group $G$ acts on a row-finite, finitely aligned multitree $E$ with no sources, and suppose that $G_v \cong \Z$ for all $v \in E^0$. The induced action $G \curvearrowright \partial E$ is topologically free if and only if for all $v \in \Gamma^0$ there exists an infinite path $\lambda = e_1 e_2 \ldots \in v\Gamma^\infty$ such that
	\[\limsup_{k\to\infty}\,\langle q(\lambda_{k-1}) / \omega_{e_k} \rangle = \infty.\]
\end{prop}

\begin{remark}\label{rem:BMPSTisWrong}
	We note that the terms $\langle q(\gamma_n) \rangle$ appearing in condition (2) of \cite[Theorem~7.5]{BMPST} are not quite correct, and should be replaced with $\langle q(\gamma_n) / \omega_{e_{n+1}} \rangle$ for respective $n \geq 1$. However, this does not affect the validity of the other results in \cite{BMPST}.
\end{remark}

To prove Proposition~\ref{prop:GBS topological freeness}, we need the following string of lemmas.

\begin{lemma} \label{lemma:GBS finite path isotropy}
	Suppose that a discrete group $G$ acts on a row-finite, finitely aligned multitree $E$ with no sources, and suppose that $G_v \cong \Z$ for all $v \in E^0$. Fix a finite path $\alpha = e_1 e_2 \dots e_n$ in $\Gamma$. For any lift $\alpha' \in E^n$ of $\alpha$, an element $m 1_{r(\alpha')} \in G_{r(\alpha')}$, $m \in \Z$ fixes $\alpha'$ if and only if $\langle q(\alpha_{k-1}) / \omega_{e_k} \rangle \mid m$ for all $1 \leq k \leq n$.
\end{lemma}

\begin{proof}
	We show by induction on $n$ that $m 1_{r(\alpha')}$ fixes $\alpha'$ if and only if $mq(\alpha_{k-1}) / \omega_{e_k}$ is an integer for all $1 \leq k \leq n$; the claim of the lemma follows from this. (Note that $\alpha_0:=r(\alpha)$.) Moreover, we show in the same induction that if $m 1_{r(\alpha')}$ fixes $\alpha'$, then $m q(\alpha)$ is an integer and $m 1_{r(\alpha')} = m q(\alpha) 1_{s(\alpha')}$ up to sign.
	
	If $n = 1$, then $\alpha = e_1$, and for any lift $e'$ of $e_1$, we have that $G_{e'}$ is the subgroup of $G_{r(e')}$ of index $\abs{\omega_{e_1}}$. Hence $m 1_{r(\alpha')}$ fixes $\alpha'$ if and only if $\omega_{e_1} \mid m$, or equivalently, if $m/\omega_{e_1}$ is an integer. Since $q(\alpha_0) = q(r(\alpha)) = 1$ the first claim holds for $n=1$. For the second claim, note that $m1_{r(e')} = (m/\omega_{e_1})1_{e'} = (m\omega_{\overline{e_1}}/\omega_{e_1})1_{s(e')}$ up to sign, and $m\omega_{\overline{e_1}}/\omega_{e_1} = mq(e_1)$ is an integer since $m/\omega_{e_1}$ is.
	
	Now suppose that the inductive claim holds for all paths of length $n = l$. Fix a finite path $\alpha = e_1 e_2 \dots e_{l+1}$ in $\Gamma$ of length $l+1$, and a lift $\alpha' = e'_1 e'_2 \dots e'_{l+1} \in E^{l+1}$ of $\alpha$. Fix $m \in \Z$. Now $m 1_{r(\alpha')}$ fixes $\alpha'$ if and only if it fixes both $\alpha'_l$ and $e'_{l+1}$. The inductive assumption gives that $m 1_{r(\alpha')}$ fixes $\alpha'_l$ if and only if $mq(\alpha_{k-1}) / \omega_{e_k}$ is an integer for all $1 \leq k \leq l$, and that in this case, $m' := mq(\alpha_l)$ is an integer and $m' 1_{s(\alpha'_l)} = m 1_{r(\alpha')}$ (up to sign) in $G$. Now from the $n=1$ case, we know that $m'1_{s(\alpha'_l)} = m'1_{r(e'_{l+1})}$ fixes $e'_{l+1}$ if and only if $m'/\omega_{e_{l+1}} = mq(\alpha_l)/\omega_{e_{l+1}}$ is an integer. This proves the first inductive claim for $n = l+1$.
    
    For the second claim, note that the inductive assumption gives that $m q(\alpha_l)$ is an integer and that $m 1_{r(\alpha')} = m q(\alpha_l) 1_{s(\alpha'_l)} = m q(\alpha_l) 1_{r(e'_{l+1})}$ up to sign. But the $n=1$ case for the edge $e'_{l+1}$ (with $m'$ in place of $m$) gives that $m'q(e_{l+1}) = mq(\alpha_l)q(e_{l+1}) = mq(\alpha)$ is an integer and that $m q(\alpha_l) 1_{r(e'_{l+1})} = mq(\alpha) 1_{s(e'_{l+1})} = mq(\alpha) 1_{s(\alpha')}$ up to sign. Thus the second claim also holds.
	
	Hence the inductive claim holds for all $n \geq 1$, and we are done.
\end{proof}

\begin{coro} \label{coro:finite path nontrivial isotropy}
    Suppose that a discrete group $G$ acts on a row-finite, finitely aligned multitree $E$ with no sources, and suppose that $G_v \cong \Z$ for all $v \in E^0$. Every finite path in $E$ has non-trivial isotropy.
\end{coro}

\begin{lemma} \label{lemma:GBS topological freeness}
	Suppose that a discrete group $G$ acts on a row-finite, finitely aligned multitree $E$ with no sources, and suppose that $G_v \cong \Z$ for all $v \in E^0$. Fix an infinite path $\lambda = e_1 e_2 \ldots \in \Gamma^\infty$ and a lift $\lambda' \in E^\infty$ of $\lambda$. We have
	\[G_{\lambda'} = 0 \quad \iff \quad \limsup_{k \to \infty}\,\langle q(\lambda_{k-1}) / \omega_{e_k} \rangle = \infty.\]
\end{lemma}

\begin{proof}
	First note that $G_{\lambda'}$ is a subgroup of $G_{r(\lambda')}$. Now, an element $m 1_{r(\lambda')}$, $m \in \Z$ fixes $\lambda'$ if and only if it fixes $\lambda'_k$ for all $k \geq 1$. By Lemma~\ref{lemma:GBS finite path isotropy} this is equivalent to the condition that $\langle q(\lambda_{k-1}) / \omega_{e_k} \rangle \mid m$ for all $k \geq 1$. If $\limsup_{k \to \infty}\,\langle q(\lambda_{k-1}) / \omega_{e_k} \rangle = \infty$, then the set $\{\langle q(\lambda_{k-1}) / \omega_{e_k} \rangle : k \geq 1\}$ is unbounded, and so the only value of $m$ satisfying $\langle q(\lambda_{k-1}) / \omega_{e_k} \rangle \mid m$ for all $k \geq 1$ is $m = 0$. Hence $G_{\lambda'} = 0$ in this case. On the other hand, if $\limsup_{k \to \infty}\,\langle q(\lambda_{k-1}) / \omega_{e_k} \rangle < \infty$, then the set $\{\langle q(\lambda_{k-1}) / \omega_{e_k} \rangle : k \geq 1\}$ is bounded. This set therefore has a (finite) least common multiple $M$, and by definition we have that $\langle q(\lambda_{k-1}) / \omega_{e_k} \rangle \mid M$ for all $k \geq 1$. Hence $M 1_{r(\lambda')} \neq 0$ fixes $\lambda'$, so $G_{\lambda'}$ is not trivial in this case.
\end{proof}

\begin{proof}[Proof of Proposition~\ref{prop:GBS topological freeness}]
	For the forward direction, suppose that the action $G \curvearrowright \partial E$ is topologically free, and fix $v \in \Gamma^0$. Choose any lift $v' \in E^0$ of $v$. Since $G \curvearrowright \partial E$ is topologically free, there is some infinite path $\lambda' \in v'E^\infty$ such that $[\lambda']$ has trivial isotropy. In particular, $G_{\lambda'} = 0$, and so writing $\lambda \in v\Gamma^\infty$ for the image of $\lambda'$ under the quotient map, Lemma~\ref{lemma:GBS topological freeness} shows that $\lambda$ satisfies the required condition.
	
	We now prove the backward direction. Fix $v' \in E^0$ and $g \in G\setminus\{1\}$; it is enough to show that there is some infinite path $\lambda' \in v'E^\infty$ such that $g$ does not fix $[\lambda']$. Write $v \in \Gamma^0$ for the image of $v'$ under the quotient map. By assumption, there is an infinite path $\lambda = e_1 e_2 \ldots \in v\Gamma^\infty$ satisfying $\limsup_{k\to\infty}\,\langle q(\lambda_{k-1}) / \omega_{e_k} \rangle = \infty$. If $g$ does not fix $[\lambda']$ for any lift $\lambda' \in v'E^\infty$ of $\lambda$, then we are done, so suppose that there is some lift $\lambda'' \in v'E^\infty$ of $\lambda$ such that $g \cdot [\lambda''] = [\lambda'']$. This means that there are $m,n \in \N$ such that $\sigma^m(g \cdot \lambda'') = \sigma^n(\lambda'')$.
	
	First we claim that $m \neq n$. Suppose for a contradiction that $m=n$. Then $g$ fixes $\sigma^n(\lambda'')$, so $\sigma^n(\lambda'')$ has non-trivial isotropy. This means that the stabiliser $G_{\sigma^n(\lambda'')}$ is some finite-index subgroup of $G_{r(\sigma^n(\lambda''))}$. Now by Corollary~\ref{coro:finite path nontrivial isotropy}, we have that $\lambda''_n$ also has non-trivial isotropy, and thus $G_{\lambda''_n}$ is also some finite-index subgroup of $G_{r(\sigma^n(\lambda''))} = G_{s(\lambda''_n)}$. Since $G_{r(\sigma^n(\lambda''))} \cong \Z$, the intersection of two finite-index subgroups is again a finite-index subgroup, and so is in particular not the trivial subgroup $\{0\}$. This means that $G_{\sigma^n(\lambda'')} \cap G_{\lambda''_n}$ is non-trivial. But an element in $G_{r(\sigma^n(\lambda''))}$ fixes $\lambda''$ if and only if it fixes both $\lambda''_n$ and $\sigma^n(\lambda'')$, so $G_{\lambda''} = G_{\sigma^n(\lambda'')} \cap G_{\lambda''_n}$. So $G_{\lambda''}$ is non-trivial, contradicting Lemma~\ref{lemma:GBS topological freeness} and the choice of $\lambda''$. Hence $m \neq n$.
	
	Now suppose that $m < n$. Writing $\mu'' = \sigma^n(\lambda'')$, we have that $g \cdot \mu'' = \sigma^{n-m}(\mu'')$. This implies that $\mu := \sigma^n(\lambda)$, which is the image of $\mu''$ under the quotient map, has the form $\eta\eta\cdots$ for some cycle $\eta$ in $\Gamma$ of length $n-m$. We claim that $\langle q(\eta) \rangle > 1$. Suppose for a contradiction that $\langle q(\eta) \rangle = 1$, so $q(\eta)$ is an integer. Write $a_k := q(\lambda_{k-1})/\omega_{e_k}$ for $k \geq 1$. Since $\sigma^n(\lambda)$ is periodic (with period $n-m$), we get that $a_{k + n-m} = q(\eta)a_k$ for all $k > n$. Since $q(\eta)$ is an integer, this implies that $\langle a_{k + n-m} \rangle \leq \langle a_k \rangle$ for all $k > n$, and so the sequence $\{\langle a_k \rangle\}_{k=1}^\infty = \{\langle q(\lambda_{k-1})/\omega_{e_k} \rangle\}_{k=1}^\infty$ is bounded, contradicting the choice of $\lambda$. Hence $\langle q(\eta) \rangle > 1$.
	
	Now we must have that $\abs{\omega_{e}} > 1$ for some edge $e$ in $\eta$; otherwise $q(\eta)$ would be an integer and hence $\langle q(\eta) \rangle = 1$, contradicting the above. By Lemma~\ref{lemma:path lifts}, this means that for every lift $w' \in E^0$ of $r(\eta)$, there is more than one lift of $\eta$ with range $w'$.
	
	We are now ready to construct an infinite path $\lambda' \in v'E^\infty$ such that $g \cdot [\lambda'] \neq [\lambda']$. Notice that $\mu''_{n-m}$ is a lift of $\eta$ with range $r(\mu'')$. By the above discussion, there is a distinct lift $\eta'$ of $\eta$ with the same range. Let $\xi'$ be a lift of $\eta\eta\cdots$ with range $s(\eta')$, and write $\lambda' = \lambda''_n \eta' \xi'$. We show that $g$ does not fix $[\eta'\xi'] = [\lambda']$. Note that $g$ sends $r(\eta') = r(\mu''_{n-m})$ to $s(\mu''_{n-m})$, since $g \cdot \mu'' = \sigma^{n-m}(\mu'')$. So $g \cdot \eta'\xi' \in s(\mu''_{n-m})E^\infty$.
	
	We claim that $Z_\partial(s(\eta')) \cap Z_\partial(s(\mu''_{n-m})) = \emptyset$; this would then imply that $\xi' \in s(\eta')E^\infty$ (and therefore $\lambda'$) is not shift-tail equivalent to $g \cdot \eta'\xi' \in s(\mu''_{n-m})E^\infty$ (and hence to $g \cdot \lambda'$). Suppose for a contradiction that $Z_\partial(s(\eta')) \cap Z_\partial(s(\mu''_{n-m})) \neq \emptyset$. Then Proposition~\ref{prop:multitree cylinder intersection} implies that $s(\eta')$ and $s(\mu''_{n-m})$ have some (minimal) common upper bound $u \in E^0$. By definition, this means that there are paths $\alpha'$ and $\alpha''$ from $u$ to $s(\eta')$ and $s(\mu''_{n-m})$ respectively. But then $\eta'\alpha'$ and $\mu''_{n-m}\alpha''$ would both be paths from $u$ to $r(\mu'') = r(\eta')$, and they are distinct since $\eta' \neq \mu''_{n-m}$ by construction and since $\eta'$ and $\mu''_{n-m}$ have the same length. This contradicts the fact that $E$ is a multitree, and hence $Z_\partial(s(\eta')) \cap Z_\partial(s(\mu''_{n-m})) = \emptyset$ as claimed. So $\lambda'$ and $g \cdot \lambda'$ are not shift-tail equivalent, meaning that $g$ does not fix $[\lambda'] \in Z_\partial(v')$, as required.
	
	Finally suppose that $m > n$. Note that
	\[\sigma^n(g^{-1} \cdot \lambda'') = g^{-1} \cdot \sigma^n(\lambda'') = g^{-1} \cdot \sigma^m(g \cdot \lambda'') = g^{-1}g \cdot \sigma^m(\lambda'') = \sigma^m(\lambda'').\]
	Since $n < m$, the argument for the $m<n$ case (with $m$ and $n$ swapped) gives the existence of an infinite path $\lambda' \in v'E^\infty$ such that $g^{-1} \cdot [\lambda'] \neq [\lambda']$. But acting by $g$ on both sides of the equation gives that $[\lambda'] \neq g \cdot [\lambda']$, which is what we needed.
\end{proof}

\subsection{Amenability} \label{subsec:amenability}

For actions on undirected trees, we have that the amenability of all vertex stabilisers ensures the amenability of the action $G \curvearrowright \partial T$ \cite[Theorem 5.29]{BMPST}. We would like something similar in the multitree context. It is unclear to the authors what the situation is in general, however we are able to prove Proposition~\ref{prop:directed tree amenability}, which is the corresponding result for actions on directed trees.

In this section we use the formulation of an amenable group action given in \cite[Definition 2.1]{AD2}, which is as follows. For any locally compact group $G$, write $\Prob(G)$ for the space of probability measures on $G$ equipped with the weak*-topology. Given $s, t \in G$, $f \in C_c(G)$ and $m \in \Prob(G)$, set
\[(s \cdot f)(t) = f(s^{-1} t) \quad \text{and} \quad (s \cdot m)(f) = m(s^{-1} \cdot f).\]
An action of $G$ on some locally compact Hausdorff space $X$ is \textit{amenable} if there is a net $(m_i)_{i \in I}$ of continuous maps $x \mapsto m_i^x$ from $X$ to $\Prob(G)$ such that
\[\lim_i \norm{s \cdot m_i^x - m_i^{s \cdot x}}_1 = 0\]
uniformly for $(x, s)$ in compact subsets of $X \times G$.

\begin{prop} \label{prop:directed tree amenability}
	Let $G$ be a discrete group acting on a directed tree $T_+$. If for all $v \in T^0$ the stabiliser $G_v \leq G$ is an amenable group, then the induced action $G \curvearrowright \partial T_+$ is amenable.
\end{prop}

We first prove the following lemma.

\begin{lemma} \label{lemma:subspace amenability}
	Let $G$ be a locally compact group, and let $X$ and $Y$ be locally compact spaces on which $G$ acts. Suppose that $G \curvearrowright X$ is amenable, and that there exists a continuous, $G$-equivariant map $\varphi\colon Y \to X$. Then the action $G \curvearrowright Y$ is amenable.
\end{lemma}

\begin{proof}
	By definition of the amenability of $G \curvearrowright X$, there is a net $(m_i)_{i \in I}$ of continuous maps $x \mapsto m_i^x$ from $X$ to $\Prob(G)$ such that
	\begin{equation} \label{eq:amenability}
	  \lim_i \norm{s \cdot m_i^x - m_i^{s \cdot x}}_1 = 0
	\end{equation}
	uniformly for $(x, s)$ in compact subsets of $X \times G$. Now since $\varphi \colon Y \to X$ is assumed to be continuous, each $m_i \circ \varphi$ is a continuous map from $Y$ to $\Prob(G)$, where $y \mapsto m_i^{\varphi(y)}$. Now, using Equation~\eqref{eq:amenability} and the $G$-equivariance of $\varphi$, we get that
	\[\lim_i \norm{s \cdot (m_i \circ \varphi)^y - (m_i \circ \varphi)^{s \cdot y}}_1 = \lim_i \norm{s \cdot m_i^{\varphi(y)} - m_i^{\varphi(s \cdot y)}}_1 = \lim_i \norm{s \cdot m_i^{\varphi(y)} - m_i^{s \cdot \varphi(y)}}_1 = 0\]
	uniformly for $(\varphi(y), s)$ in compact subsets of $X \times G$. But $\varphi \times \id_G$ maps any compact set in $Y \times G$ to a compact subset of $X \times G$, since continuous images of compact sets are compact and $\varphi \times \id_G$ is continuous. Hence the limit is uniform for $(y, s)$ in compact subsets of $Y \times G$, and so $G \curvearrowright Y$ is amenable.
\end{proof}

\begin{proof}[Proof of Proposition \ref{prop:directed tree amenability}]
	Write $T$ for the underlying undirected tree of $T_+$. The undirected tree $T$ inherits an action of $G$ from $T_+$, with the same vertex stabilisers. In particular, \cite[Theorem 5.29]{BMPST} implies that the action $G \curvearrowright \partial T$ is amenable. We aim to apply Lemma \ref{lemma:subspace amenability} with $X = \partial T$ and $Y = \partial T_+$.
	
	For $\mu \in T_+^\infty$ we write $[\mu]_+$ for the shift-tail equivalence class of $\mu$ in $\partial T_+$ to distinguish it from the shift-tail equivalence class $[\mu]$ in $\partial T$. Define a map $\varphi\colon \partial T_+ \to \partial T$ by $\varphi([\mu]_+) = [\mu]$ for all $\mu \in T_+^\infty$. This map is well-defined since for any $\mu_1, \mu_2 \in T_+^\infty$, we have $[\mu_1]_+ = [\mu_2]_+ \implies [\mu_1] = [\mu_2]$. Moreover, $\varphi$ is $G$-equivariant since for all $g \in G$ and $\mu \in T_+^\infty$ we have
	\[\varphi(g \cdot [\mu]_+) = \varphi([g \cdot \mu]_+) = [g \cdot \mu] = g \cdot [\mu] = g \cdot \varphi([\mu]_+).\]
	
	We now show that $\varphi$ is continuous. Recall that $\partial T$ has a basis of (compact) open sets $\{Z_\partial(e) : e \in T^1\}$, so it is enough to show that $\varphi^{-1}(Z_\partial(e))$ is open in $\partial T_+$ for all $e \in T^1$. Fix $e \in T^1$. We have
	\begin{align*}
		\varphi^{-1}(Z_\partial(e)) &= \{[\mu]_+ : \mu \in T_+^\infty,\ [\mu] \in Z_\partial(e)\} \\
		&= \{[\mu]_+ : \mu \in T_+^\infty,\ \exists \lambda \in s(e)T^\infty_r\ \text{such that}\ [\mu] = [\lambda]\ \text{and}\ \lambda_1 \neq \overline{e}.\}
	\end{align*}
	Now recall that $[\mu] = [\lambda]$ means that there are some $m,n \in \N$ and edges $e_i \in T^1$ for $1 \leq i \leq n$ such that $\lambda = e_1 e_2 \dots e_n \sigma^m(\mu)$. Note that we can write $\sigma^m(\mu) = e_{n+1} \sigma^{m+1}(\mu)$ where $e_{n+1} \in T_+^1$, and so $\lambda$ has the form 
	\[
	\lambda = e_1 e_2 \dots e_n e_{n+1} \sigma^{m+1}(\mu). 
	\]
	In particular, relabelling $n+1$ as $n$, we can assume without loss of generality that $n \geq 1$ and that $e_n \in T_+^1$. Moreover, since $\lambda$ is a reduced infinite path, we get that $\alpha := e_1 e_2 \dots e_n$ is also reduced. Note also that $\mu' := \sigma^{m+1}(\mu)$ is an infinite directed path. So we can write
	\begin{align*}
		&\varphi^{-1}(Z_\partial(e))\\ 
		&\qquad= \{[\mu']_+ : \mu' \in T_+^\infty,\ \exists \alpha = e_1 e_2 \dots e_n \in s(e)T^*_r\ \text{such that}\\
        &\hspace{15em} n \geq 1, \ e_1 \neq \overline{e},\ e_n \in T_+^1\ \text{and}\ \mu' \in s(\alpha)T_+^\infty\} \\
		&\qquad= \bigcup_{\substack{\alpha = e_1 \dots e_n \in s(e)T^*_r, \\ n \geq 1,\ e_1 \neq \overline{e},\ e_n \in T_+^1}} Z_\partial(s(\alpha)),
	\end{align*}
	which implies that $\varphi^{-1}(Z_\partial(e))$ is open. Hence $\varphi$ is continuous, so we can apply Lemma \ref{lemma:subspace amenability} to conclude that the action $G \curvearrowright \partial T_+$ is amenable.
\end{proof}

\bibliographystyle{plain}
\bibliography{gogk}

\end{document}